\documentclass[12pt]{amsart}
\usepackage[english]{babel}
\usepackage{amsmath}
\usepackage{amsfonts}
\usepackage{amssymb}
\usepackage{amsthm}
\usepackage{graphicx}
\usepackage[T1]{fontenc}
\usepackage[utf8x]{inputenc}
\usepackage{color}
\usepackage{comment}
\usepackage{tikz}
\usepackage{hyperref}
\hypersetup{
  colorlinks=true,    
  breaklinks=true,    
   citecolor= blue,
   linkcolor= blue,    
  pagebackref=true,     
    }

\newcommand{\Anne}{\color{red}{\bf Anne:}}

\newcommand{\todoii}[1]{{\textbf{[#1]}}} 
\newcommand{\todomargii}[1]{\marginpar{#1}} 
\newcommand{\todoBii}[1]{\vspace{5mm}\par \noindent 
\framebox{\begin{minipage}[c]{0.95 \textwidth} \tt #1\end{minipage}} \vspace{5mm} \par}

\newcommand{\compl}[1]{\footnote{\color{magenta}#1}} 

\renewcommand{\todoii}[1]{} 
\renewcommand{\todomargii}[1]{}
\renewcommand{\todoBii}[1]{}
\renewcommand{\compl}[1]{} 

\theoremstyle{plain}
\newtheorem{Theorem}{Theorem}[section]
\newtheorem{Theorem?}{Theorem(?)}[section]
\newtheorem{Proposition}[Theorem]{Proposition}
\newtheorem{Corollary}[Theorem]{Corollary}
\newtheorem{Lemma}[Theorem]{Lemma}

\theoremstyle{definition}
\newtheorem{Example}[Theorem]{Example}
\newtheorem{Definition}[Theorem]{Definition}
\newtheorem{Remark}[Theorem]{Remark}

\DeclareMathOperator{\id}{id}

\DeclareMathOperator{\CAT}{CAT}
\newcommand{\ov}[1]{\overline{#1}}
\newcommand{\isomto}{\overset{\sim}{\rightarrow}}
\renewcommand{\vec}[1]{\overrightarrow{#1}}

\newcommand{\calL}{\mathcal L}

\newcommand{\calG}{\mathcal G} 
\newcommand{\calDG}{\mathcal {DG}} 
\DeclareMathOperator{\supp}{supp}

\newcommand{\n}{n} 

\newcommand{\R}{{\mathbb R}} 
\newcommand{\N}{{\mathbb N}} 
\newcommand{\bN}{{\mathbb N}} 
\newcommand{\bZ}{{\mathbb Z}} 

\newcommand{\abs}[1]{|#1|}
\newcommand{\norm}[1]{||#1||}
\newcommand{\wt}[1]{\widetilde{#1}}

\newcommand{\PP}{\mathbb{P}} 

\DeclareMathOperator{\PSL}{PSL}

\newcommand{\Ga}{\Gamma}

\newcommand{\Sigmat}{\widetilde{\Sigma}} 
\newcommand{\Sigmatc}{\widetilde{\Sigma}^c} 

\newcommand{\subS}{\Sigma'}

\newcommand{\met}{m}

\newcommand{\W}{W} 
\newcommand{\Wt}{{\widetilde{\W}}} 
\newcommand{\ct}{{\widetilde{c}}} 
\newcommand{\Wtc}{{\widetilde{W}^c}}  

\newcommand{\Pt}{{\wt{P}}} 
\newcommand{\Ppt}{{\wt{P'}}} 
\newcommand{\calP}{\mathcal P} 

\DeclareMathOperator{\Hom}{Hom}
\DeclareMathOperator{\Out}{Out}
\newcommand{\Teich}{\mathcal{T}}
\newcommand{\Char}{\Xi} 
\newcommand{\CharMax}{\Char_{\text{Max}}}
\newcommand{\CharHit}{\Char_{\text{Hit}}}
\newcommand{\CharDot}{\mathcal{X}} 

\DeclareMathOperator{\Flat}{\vec{\rm Flat}} 
\DeclareMathOperator{\FlatDLR}{{Flat}} 
\DeclareMathOperator{\Mix}{\vec{{\rm Mix}}} 
\DeclareMathOperator{\MixDLR}{{Mix}} 
\newcommand{\compC}[1]{\overline{#1}^{\rm WL}}
\newcommand{\compN}[1]{\overline{#1}^{N}} 
\newcommand{\bord}{\partial}
\newcommand{\bordC}{\partial^{\rm WL}}

\newcommand{\dCharWL}{\partial^{\rm WL}\Char}

\newcommand{\dCharMaxWL}{\partial^{\rm WL}\CharMax}

\newcommand{\CharHitWL}{\CharHit^{\rm WL}}

\newcommand{\CharDotWL}{\CharDot^{\rm WL}}
\newcommand{\dCharDotWL}{\partial^{\rm WL}\CharDot}

\newcommand{\MCG}{{\rm MCG}}

\newcommand{\vd}{\vec{d}} 

\renewcommand{\L}{L} 
\newcommand{\vL}{\vec{L}} 
\newcommand{\CalvL}{\vec{\calL}} 
\newcommand{\PCalvL}{\PP\vec{\calL}} 

\DeclareMathOperator{\Syst}{Syst}

\newcommand{\vnu}{\vec{\nu}} 
\newcommand{\ML}{\mathcal{ML}} 
\newcommand{\Fol}{\mathcal{F}} 
\newcommand{\Folt}{\wt{\mathcal{F}}} 
\newcommand{\Lam}{\Lambda} 
\newcommand{\mLam}{\lambda} 
\newcommand{\vmLam}{\vec{\mLam}} 
\newcommand{\Lamt}{\wt{\Lam}} 
\newcommand{\Lamot}{\wt{\Lam_0}} %
\newcommand{\mLamt}{\wt{\mLam}} 

\newcommand{\Sigmab}{\Sigma^b} 
\newcommand{\Sigmabt}{\wt{\Sigmab}} 
\newcommand{\Lamb}{\Lam^b} 
\newcommand{\Lamtb}{\Lamt^b}  
\newcommand{\nub}{\nu^b} 
\newcommand{\vnub}{\vnu^b} 
\newcommand{\mLamb}{\mLam^b} 
\newcommand{\vmLamb}{\vmLam^b} 
\newcommand{\metb}{\met^b} 
\newcommand{\subSb}{\Sigma^{'b}} 
\newcommand{\Eeb}{\Ee^b} 
\newcommand{\Eebt}{\wt{\Ee}^b} 

\newcommand{\K}{K} 

\newcommand{\M}{M} 

\newcommand{\Pieces}{\mathcal{P}}

\newcommand{\piXM}{\pi} 

\newcommand{\Aa}{\mathfrak{a}}
\newcommand{\Cc}{{\overline{\mathfrak a}^+}}

\newcommand{\opp}{\text{opp}} 

\newcommand{\Ee}{\mathcal E} 
\newcommand{\Eet}{\wt{\Ee}} 
\newcommand{\Vv}{\mathcal V} 

\begin{document}
	\title[$\PSL(2,\R)\times\PSL(2,\R)$ maximal representations]{Weyl 		chamber length compactification of  the $\PSL(2,\R)\times\PSL(2,\R)$ maximal character variety}

\author[]{M. Burger}
\address{Department Mathematik, ETH Zentrum, 
	R\"amistrasse 101, CH-8092 Z\"urich, Switzerland}
\email{burger@math.ethz.ch}

\author[]{A. Iozzi}
\address{Department Mathematik, ETH Zentrum, 
	R\"amistrasse 101, CH-8092 Z\"urich, Switzerland}
\email{iozzi@math.ethz.ch}

\author[]{A. Parreau}
\address{Institut Fourier, CS 40700, 38058 Grenoble cedex 09, France}
\email{Anne.Parreau@univ-grenoble-alpes.fr}

\author[]{M. B. Pozzetti}
\address{Mathematical Institute, Heidelberg University, Im Neuenheimerfeld 205, 69120 Heidelberg, Germany }
\email{pozzetti@mathi.uni-heidelberg.de}
\thanks {Beatrice Pozzetti is supported by the Deutsche Forschungsgemeinschaft under Germany’s Excellence Strategy EXC-2181/1 - 390900948 (the Heidelberg STRUCTURES Cluster of Excellence), and acknowledges further support by DFG grant 338644254 (within the framework of SPP2026).
}

\setcounter{tocdepth}{2}
\newpage
\maketitle

\begin{abstract}
  We study  the vectorial length compactification  of the space of conjugacy classes of maximal representations  of the fundamental group $\Gamma$  of a closed hyperbolic surface $\Sigma$  in  $\PSL(2,\R)^n$. 
  We identify the boundary with the sphere $\PP((\ML)^n)$,
  where $\ML$ is the space of measured geodesic laminations on $\Sigma$. 
  In the case $n=2$, we give a geometric interpretation of the boundary as
  the space of homothety classes of { $\R^2$-mixed structures} on $\Sigma$.
  We associate to such a structure a dual tree-graded space 
  endowed with an $\R_+^2$-valued metric,
  which we show to be universal with respect to actions on products of two $\R$-trees with the given length spectrum. 
\end{abstract}

\tableofcontents

\section{Introduction}
Let $\Sigma$ be a connected closed oriented  surface of genus $g\geq 2$
and let $\Gamma=\pi_1(\Sigma)$ be its fundamental group.
%
For a real reductive Lie group $G$, we denote by
$\Char(\Gamma,G):=\Hom_{red}(\Gamma,G)/G$ the \emph{character variety},
namely the space of conjugacy classes of completely reducible
representations $\rho\colon \Gamma\to G$.
The third named author used Weyl chamber valued length functions to
construct the {\em Weyl chamber length compactification} of the character variety
\cite{APcomp}, a compactification that generalizes Thurston's
compactification of the Teichm\"uller space $\Teich=\Teich(\Sigma)$,
for which $G=\PSL(2,\R)$.
The boundary $\dCharWL(\Gamma, G)$ of the
character variety  
in the Weyl chamber length compactification is a compact subset of the set
$\PP(\Cc^\Gamma)$ of  homothety classes of non-zero functions
$\vL\colon \Gamma\to \Cc$, where $\Cc$ is a fixed closed Weyl chamber of
$G$.  Boundary points may be interpreted as projectivized $\Cc$-length
functions of actions on $\R$-buildings \cite{APcomp}.

In the case of $G=\PSL(2,\R)$, we know much more: Thurston proved that the boundary of 
this compactification of Teichm\"uller space is the
projectivization $\PP (\ML)$ of the space $\ML$ of measured geodesic laminations on $\Sigma$,
realized as the cone  $\ML\subset \R_+^\Gamma$ of functions
 $\Gamma\to\R_+$ that are intersection functions of measured geodesic
 laminations (or  equivalently, of measured foliations) on
 $\Sigma$ \cite{FLP, Bonahon88}. This allowed him to prove that this boundary is a sphere of dimension
 $6g-7$ and the compactification is a closed ball \cite{FLP}.
 Boundary points may also be interpreted as projectivized length functions of actions on
 real trees \cite{Bestvina, Paulin, Morgan-Shalen}. 

%

If $G$ is of Hermitian type, that is if the associated symmetric space
admits a $G$-invariant complex structure, the character variety
contains a generalisation of the Teichm\"uller space, the space of
\emph{maximal representations}, which we denote by $\CharMax(\Gamma,G)$. 
We refer the reader to
\cite{BILW, BIW} for the theory of maximal representations.
Maximal representations share many features with the subset of the character variety consisting of Hitchin representations, 
which are defined for $G$ real split, and whose character variety we denote by $\CharHit(\Gamma, G)$. 
%
%
%
In \cite{BIPP19, BIPP21} we used geodesic currents to study the Weyl
chamber length compactification $\CharDotWL(\Gamma, G)$ of $\CharDot(\Gamma, G)$, 
where $\CharDot(\Gamma, G)$ denotes either the Hitchin or the maximal character variety.
We showed that, as soon as the group $G$ has higher rank, the mapping
class group admits a non-empty open domain of discontinuity for its
action on the boundary $\dCharDotWL(\Gamma,G)$, the so-called
\emph{positive systole} subset; moreover for a dense subset of
boundary points, the associated length function can be computed as
intersection with a weighted multicurve.
%
Natural questions arise: 
to determine the topology of 
$\CharDotWL(\Gamma,G)$; to interpret boundary points in terms of
geometric structures on the surface, in particular in the positive
systole subset; and to relate such geometric structures with the
associated actions on $\R$-buidings, for example by finding a nice
invariant subset following the third named author's work in the case of
$G=\PSL(3,\R)$
\cite{APdegFG}.


In this text,  we address all the questions introduced above in the case of $G=\PSL(2,\R)^n$ for
 $n\geq 2$. The product structure of $G$ lets us identify the space $\CharMax(\Gamma,G)$   with the product $\Teich^n$ of $n$ copies of the
Teichmüller space $\Teich=\CharMax(\Gamma,\PSL(2,\R))$. 
A model Weyl chamber for $G$ is $\Cc=\R_+^n$, hence boundary points
are homothety classes of
non-zero functions $\vL\colon \Gamma\to \R_+^n$.  Such functions 
can be identified with $n$-tuples  $(L_i)_{i=1,\ldots,n}$ of 
functions $L_i\colon \Gamma\to \R_+$.

We identify the boundary of $\Teich^n$  in its vectorial length compactification with the projectivisation of $(\ML)^n$;
in particular  we show that, being a join of $n$ spheres, it is  topologically a sphere of dimension $n(6g-6)-1$. 
This also allows us to give a precise description of the open domain of discontinuity for the mapping class group alluded to before; 
for this we denote by  $(\ML)^n_{>0}$ the set consisting of $n$-tuples of measured laminations with \emph{positive joint systole}, 
namely $n$-tuples for which the function $\sum_{i=1}^n L_i(\gamma)$ has a positive lower bound on $\Gamma$.

\begin{Theorem}
\label{thmintro-IdBoundary}
The  boundary $\dCharMaxWL(\Ga,\PSL(2,\R)^n)$  is
$\PP((\ML)^n)$, and $\MCG(\Sigma)$ acts properly discontinuously on $\PP((\ML)^n_{>0})$.
\end{Theorem}
We refer the reader to \S\ref{s-Boundary}, and in particular to Corollary~\ref{c-Wlength} for applications of this result 
to the study of compactifications induced by length functions on $\Gamma$ 
associated to various $\PSL(2,\R)^n$-invariant distance functions on the product $(\mathcal H^2)^n$ of $n$ hyperbolic planes. 

\medskip

Inspired by the work of Duchin--Leiniger--Rafi \cite{DLR} and of Morzadec \cite{Mor}, we  give a geometric interpretation of boundary points in the case of $n=2$.
Duchin--Leiniger--Rafi introduced mixed structures to give a geometric compactification of the space of $\FlatDLR(\Sigma)$ of half translation structures on a compact surface $\Sigma$ up to isometry.
For the purposes of this paper we define an {\em $\R^2$-mixed
  structure}\footnote{See Remark~\ref{r-DLRmix} in \S\ref{s-Mix} for more details as well as a
  comparison with the structures considered in \cite{DLR} and in
  \cite{OT0,OT1}} on $\Sigma$
 as a triple
$\M=(\subS,\K,\vmLam)$ where $\subS$ is an open geodesic subsurface of
$\Sigma$, $\K$ is a half-translation structure on $\subS$, extendible
at boundary components regarded as punctures,   and
$\vmLam=(\Lam, \nu_1,\nu_2)$ is a {\em $2$-measured geodesic lamination}, that
is geodesic lamination $\Lam$ disjoint from $\subS$ and endowed with a pair of transverse
measures  $(\nu_1,\nu_2)$, with 
$\supp(\nu_1)\cup \supp(\nu_2)=\Lam$.
%
%
Note that the subsurface $\subS$ or the lamination $\Lambda$ may be
empty, so the space $\Mix(\Sigma)$ of mixed structures on $\Sigma$
contains both the space  $\Flat(\Sigma)$ of half-translation structures (up to
isotopy) and the space of $2$-measured geodesic laminations.
The  half-translation
structure is equipped with a natural pair
$(\Fol_1,\Fol_2)$ of transverse vertical/horizontal measured
foliations. 
%
As a result a mixed structure $\M$ defines a natural pair $(\mLam_{\M,1},\mLam_{\M,2})$ of
associated ``vertical/horizontal'' measured geodesic laminations on $\Sigma$. 
Taking intersection functions, we obtain a map
$$I\colon  \Mix(\Sigma) \to \ML\times \ML$$
extending the natural map $\Flat(\Sigma) \to \ML\times \ML$. We  show in Proposition~\ref{prop-MLxMLtoMix} that $I$ is a bijection, under which $\Flat(\Sigma)$ corresponds to $\ML^2_{>0}$. Together with Theorem~\ref{thmintro-IdBoundary} this implies the following:
\begin{Theorem}
	\label{thmIntro1} 
	Let $\vL\colon \Gamma\to \R_+^2$ represent a point in the Weyl chamber length boundary of $\CharMax(\Gamma,\PSL(2,\R)^2)$. 
	Then 
	\begin{enumerate}
		\item $\vL$ is the $\R^2$-length
		function of a  unique 
		$\R^2$-mixed structure $\M=(\subS,\K,\vmLam)$ on $\Sigma$.

		\item Every non-empty
		$\R^2$-mixed structure on $\Sigma$ arises in this
		way.
	\end{enumerate}
\end{Theorem}


\medskip

We now turn to the question of relating the geometric structures
appearing in the boundary of $\CharMax(\Gamma,G)$ to the
associated actions on $\R$-buidings and  finding  nice
invariant subsets.
In the case of $G=\PSL(2,\R)^n$, the $\R$-buildings appearing in the boundary
of $\CharMax(\Gamma,G)$ are products of $n$ $\R$-trees, and hence boundary
points are homothety classes of
$\Cc$-length functions of actions of $\Gamma$ on a product
of $n$ $\R$-trees.

Notice that in his thesis \cite{Mor} Morzadec used tree-graded spaces to obtain a
geometric compactification of the space of flat structures, and relate
them with the mixed structures of \cite{DLR}.
Also, in the case of $G=\PSL(3,\R)$ and $\Gamma$ a punctured surface
group, 
the third named author associates in \cite{APdegFG}
to  large families of boundary points of  $\CharHitWL(\Gamma,\PSL(3,\R))$ 
explicit finite $\Cc$-simplicial complexes
whose universal cover, a  tree-graded space with flat surface
pieces, embbeds  equivariantly in the building preserving the natural $\Cc$-metric.

Our last result takes these perspectives in the study of the compactification of
the space of $\PSL(2,\R)\times\PSL(2,\R)$-maximal representations.
To be more precise,
a {\em tree-graded $\R^2$-space $(X, d_1,d_2)$}, is for us a space $X$ endowed with a pair 
$(d_1,d_2)$ of pseudometrics such that $d:=d_1+d_2$ is a metric, the {\em $\ell^1$-metric}, with respect to which $X$ is tree-graded  in the sense
of \cite{DrSa05}; 
we furthermore require that $X$  admits an action of $\Gamma$ preserving each
pseudometric. 
The  {\em $\R^2$-length} in $X$ of $\gamma\in\Gamma$ is then the pair
$$\vL_{X}(\gamma)=(L_{d_1}(\gamma),L_{d_2}(\gamma)$$
of translation lengths 
$$L_{d_i}(\gamma)=\inf_{x\in X}d_i(x,\gamma x)$$
of $\gamma$ for the pseudometric $d_i$.

We associate to each  mixed structure $\M$  an $\R^2$-tree graded space $(X_M, d_1,d_2)$ dual  to
$\M$,  whose pieces are either flat surfaces 
or $\R$-trees: 



\begin{Theorem}
\label{thmIntro} 
Let $\M=(\subS,\K,\vmLam)$ be an $\R^2$-mixed structure  on $\Sigma$ and  $\vL\colon \Gamma\to \R_+^2$ the associated $\R^2$-length
function. 
Then there is a tree-graded $\R^2$-space $X_M$ with length spectrum $\vL\colon \Gamma\to \R_+^2$, and satisfying the following universal property:
for any two actions of $\Gamma$ on  
$\R$-trees  $T_1,T_2$ with  length function
$(L_{T_1},L_{T_2})=\vL$, there is an equivariant embedding 
$$f\colon X_\M \mapsto T_1\times T_2$$ 
preserving each pseudometric $d_i$
(and  isometric for the $\ell^1$-metric).
\end{Theorem}

In particular, for any point in the boundary
$\dCharMaxWL(\Ga,\PSL(2,\R)^2)$, with associated mixed structure $\M$,
the corresponding action on a $\R$-building (here a product of two
trees) preserves an $\Cc$-isometrically equivariantly embbeded copy of
$X_\M$.
%

We construct the space $X_M$ by gluing
the dual trees corresponding to the lamination components, endowed
with the pair of pseudometrics induced by the pair of
transverse measures,  together with the
completions of the flat components, endowed  with the natural
pair of horizontal and vertical pseudometrics  
(see Section~\ref{s-DualTreeGradedSpace} for more details). We 
show that that the $\ell^1$-metric on $X_\M$ corresponds to the path metric induced by the $\ell^1$-norm on the flat
components (Lemma~\ref{lem-L=L1+L2Flat}).

A more analytic, and independent approach to compactifications of the space of maximal representations in rank 2 
Hermitian Lie groups has been independently pursued by Ouyang \cite{Ouy}, Ouyang--Tamburelli \cite{OT1,OT2}, and Martone--Ouyang--Tamburelli \cite{MOT}. 
In their work, they consider the length spectrum of the negatively curved metric induced on the unique invariant minimal surface associated to a maximal representation. 
They obtain in this way a compactification that they denote $\ov{{\rm Ind}(S)}$. In the case of $G=\PSL(2,\R)\times \PSL(2,\R)$, 
Martone--Ouyang--Tamburelli \cite{MOT} also consider a refinied compactification dominating $\ov{{\rm Ind}(S)}$, 
which they prove is a ball of dimension $12g-12$. 
Their notion of \emph{$(\ov{A_1^+\times A_1^+},2)$-mixed structure} introduced in this work is equivalent to our of $\R^2$-mixed structure defined above.
Moreover they prove that also the boundary of $\CharMax(\Gamma,\PSL(2,\R)^2)$ in their refined compactification can be identified with such structures. 

\textbf{Structure of the paper:}
The first two sections largely consist of preliminaries: we  recall in \S\ref{s-Comp} the construction of the Weyl chamber length
compactification for  $G=\PSL(2,\R)^n$, and we discuss in \S\ref{sec:PML} the relation between measured foliations, laminations, trees and geodesic currents. 
The only new material in  \S\ref{sec:PML} is in \S\ref{ss-Rntree} 
where we introduce the key notion of the $\R^n$-tree dual to an $n$-measured geodesic lamination.
With these at hand we prove Theorem
\ref{thmintro-IdBoundary} in \S\ref{s-Boundary}.
We introduce and study mixed structures in \S\ref{sec:mix}, where we also prove Theorem~\ref{thmIntro1}.
We construct in \S\ref{s-DualTreeGradedSpace}  the  $\R^2$-tree-graded space $(X_\M, d_1,d_2)$
dual  to a mixed structure $\M$. We prove in \S\ref{s-EmbeddingInProductOfTrees} that this tree graded space embeds isometrically  in any product of trees (Proposition~\ref{prop-EmbeddingInProductOfTrees}), which concludes the proof of Theorem~\ref{thmIntro}.







\section{The Weyl chamber length compactification}
\label{s-Comp}
The reference for this section is \cite{APcomp}.
%

\subsection{Weyl chamber and $\R^n$-valued metrics}
\label{s-wmetrics}
Given a finite reflection group $(\Aa,W)$ and a fixed model closed Weyl chamber $\Cc$,
we define a \emph{$\Cc$-valued pseudometric} on a space $X$ as a function 
$\vd \colon X \times X \to \Cc$ satisfying $\vd(x,x)=0$ and
\begin{description}
\item[Triangular inequality] $\vd(x,z)\leq \vd(x,y)+\vd(y,z)$;
\item[$\opp$-symmetry] $\vd(y,x)=\vd(x,y)^\opp$,
\end{description}
where $\opp$ is the opposition involution on $\Aa$, and $\Aa$ is
endowed with the partial order with positive cone the
Euclidean dual  $\Cc^*$ of $\Cc$.
We will call $\vd$ a {$\Cc$-valued metric} if in addition
it is separated, namely $\vd(x,y)=0$ implies $x=y$.

In this article we are interested only in the
 reflection group associated to
the semisimple Lie group $G=\PSL(2,\R)^n$, which is
$\Aa=\R^n$ with Weyl group $W=\{\pm\id_\R\}^n$.
The model closed Weyl chamber is $\Cc=\R_+^n$.
In this case the involution opposition is trivial,
$\Cc^*=\Cc=\R_+^n$ and a $\Cc$-valued (pseudo)metric amounts simply
to a $n$-tuple
\[\vd=(d_i)_{i=1,\ldots,n}\]
of usual $\R$-valued (pseudo)metric $d_i$. This gives rise to the notion of a  {\em $\R^n$-(pseudo)metric}:
\begin{Definition}\label{d.Rnmetric}
An \emph{$\R^n$-pseudometric space} is a set $X$ together with an $n$-tuple $\vd=(d_i)_{i=1,\ldots,n}$ of pseudometrics. We say that $(X,\vd)$ is an \emph{$\R^n$-metric space }if $d:=\sum_{i=1}^n d_i$ is a metric.
\end{Definition}	
Observe that even if $(X,\vd)$ is an $\R^n$-metric space, it is possible that none of the pseudo-metrics $d_i$ is separated.
\begin{Example}
	Examples of $\R^n$-metric spaces are products 
	$$X=X_1\times\ldots\times X_n$$ 
	of ordinary metric spaces, as well as subspaces thereof with the induced metric. Another example that will play an important role in our paper are the $\R^2$ metric naturally induced on the universal cover of a half translation surface, see \S\ref{s-Flat} for details. Another relevant example is given by simplicial trees with $\R^2$-valued edge lengths. This is discussed, and generalized, in \S\ref{ss-Rntree}.
\end{Example}	

In our case of interest\footnote{This is true for general finite reflection groups, but considerably harder to prove \cite{APcd}.} 
it is easy  to prove that, for any  $W$-invariant norm $N\colon \Aa=\R^n\to \R_+$,  $N$ is not decreasing on each variable on $\Cc=\R_+^n$, that is for the partial order introduced above.
As a result, the $\Cc$-valued (pseudo)metric
$\vd$ on $X$ induces an  $\R$-valued (pseudo)metric by setting
\[d_N(x,y):=N(\vd(x,y)) \;.\]
Note that $\vd(x,y)=0$ if and only if  $d_N(x,y)=0$, in particular $\vd$ is separated if and only
if $d_N$ is separated.

Norms $N\colon \R^n\to \R$ of particular interest to us are the $\ell^2$ and
$\ell^1$ norms, which we denote, respectively, by $\|\cdot\|_2$ and $\|\cdot\|_1$. They give the associated \emph{Euclidean  (pseudo)metric}
\[d_{\|\cdot\|_2}(x,y):=\sqrt{\sum_{i=1}^n d_i(x,y)^2}\]
and
\emph{$\ell^1$-(pseudo)metric}
\[d^1(x,y):=\sum_{i=1}^n d_i(x,y)\]
on $X$. We will also denote $d_{\|\cdot\|_2}$ by $\lVert \vd \rVert_2$
and $d^1$ by $\| \vd \|_1$.

\begin{Remark}
If $X=X_1\times\ldots\times X_n$ is a product, then $d_{\|\cdot\|_2}$ is the usual product metric, in particular $(X,d_{\|\cdot\|_2})$ is CAT(0) if and only if each of the $X_i$ are. However, in general,  $(X,d_{\|\cdot\|_2})$ is not a geodesic metric space. For example if $X$ is a simplicial tree with $\R^2$-valued edge lengths, the induced metric space $(X,d_{\|\cdot\|_2})$ is geodesic if and only if the two distance functions $d_i$ are proportional.
\end{Remark}	

\subsection{The Weyl chamber valued length function on $\PSL(2,\R)^n$}

The symmetric space associated to $\PSL(2,\R)^n$ is the product $X=(\mathcal H^2)^n$ of
$n$ copies of the hyperbolic plane $\mathcal H^2$. On $X$ the natural
Weyl-chamber valued distance 
$\vd \colon X \times X \to \R_+^n$ is
simply the product distance 
\[\vd(x,y)=(d(x_i,y_i)_i)_{i=1,\ldots,n}\;.\]

The {\em Weyl chamber length} (or {\em  $\R^n$-length}) of $g=(g_i)_{i=1,\ldots,n}\in\PSL(2,\R)^n$ 
can be defined as 
\begin{equation}\label{e.wL}
	\vL(g):=\inf_{x\in X}\vd(x,gx)
	\end{equation}
where the infimum is considered with respect to the partial order on $\R^n$  with positive cone
$\R_+^n$. 
As the metric space $X$ is a product, it boils down to taking the
list  of the usual lengths in each factor:
$$\vL(g)=(\L(g_i))_{i=1,\ldots,n}.$$
Here $\L(g):=\inf_{x\in \mathcal H^2}d(x,gx)$ is the usual translation length
of $g$ on the hyperbolic plane $\mathcal H^2$. This ensures in particular that the infimum in \eqref{e.wL} exists.

%

The {\em length function} of a representation $\rho\colon \Gamma\to \PSL(2,\R)$
is
$$L_\rho:=\L\circ \rho\colon \Gamma\to\R_+ \;.$$
Similarly, the {\em Weyl chamber length function} (or {\em
  $\R^n$-length function})
of  a representation $\rho\colon \Gamma\to \PSL(2,\R)^n$ is
 $$\vL_\rho:=\vL\circ \rho\colon  \Gamma\to\R_+^n \;. $$
Of course when  $\rho=(\rho_i)_{i=1,\ldots,n}$, it holds 
$\vL_\rho=(L_{\rho_i})_{i=1,\ldots,n}$.

\subsection{The Weyl chamber length compactification}
We endow the space $(\R_+^n)^\Gamma$ of functions
$\vL\colon \Gamma\to\R_+^n$  with the topology of pointwise
convergence, and the space
 $\PP((\R_+^n)^\Gamma)$ of homothety classes of non-zero functions 
$\vL$
 with the quotient topology.
The map
$$\begin{array}{cccc}
\CalvL:&\Hom(\Ga,\PSL(2,\R)^n) &\to& (\R_+^n)^\Gamma\\
&\rho &\mapsto& \vL\circ\rho  
\end{array}$$
induces a continuous map 
$$\PCalvL\colon \Teich^n \to \PP((\R_+^n)^\Gamma)$$
equivariant with respect to the mapping class group action.
The injectivity of 
$\Teich \to \PP\R_+^\Gamma$ implies that the map $\PCalvL$ is  injective as well.
The {\em Weyl chamber length compactification} 
$\compC{\Teich^n}$ of $\Teich^n$
is by definition the closure  of its image 
(which is a compact set). 
We denote
\[\bordC(\Teich^n):=\compC{\Teich^n}-\PCalvL(\Teich^n)\subset \PP((\R_+^n)^\Gamma)\]
its boundary.

\subsection{Other length compactifications}
Recall from Section~\ref{s-wmetrics} that any $W$-invariant norm  $N\colon \R^n\to \R$ induces a metric $d_N:=N(\vd)$ on $X=(\mathcal H^2)^n$, and thus a corresponding 
 length function, the {\em $N$-length} function $L_{d_N}$.  
We use this to construct the {\em $N$-length compactification}
$\compN{\Teich^n}$ of $\Teich^n$, associated to $L_{d_N}$.
%
%
%

It is easy to verify that, since $N$ is not decreasing on each variable on $\Cc=\R_+^n$,  and  $\Teich^n$ is a product, it holds 
  $L_{d_N}(g)=N(\vL(g))$.
The Weyl chamber length compactification then dominates naturally
$\compN{\Teich^n}$:
the restriction of the natural  map
\[
  \begin{array}[t]{rll}
    (\R_+^n)^\Gamma &\to& \R_+^\Gamma\\
    \vL&\mapsto& N\circ\vL
  \end{array}
\]
induces a continuous $\Out(\Gamma)$-equivariant surjective map
\[\compC{\Teich^n} \to \compN{\Teich^n} \]
restricting to identity on $\Teich^n$, see \cite{APcomp}.

\section{Trees dual to measured geodesic laminations}\label{sec:PML}

We refer the reader to \cite[Chapter 11]{Kap09} for
preliminaries on measured laminations, 
measured foliations, and the identification between equivalence classes of measured
foliations and measured geodesic laminations on hyperbolic surfaces. The material in Sections~\ref{ss-ml}, \ref{ss-gc}, and \ref{ss-dt} is classical, while the viewpoint in Section~\ref{ss-Rntree} is new, and important for our paper. 

\subsection{Measured foliations}\label{ss-mf}
Let $\Fol$ be transverse measured foliation on
on a topological surface $\Sigma$
with fundamental group $\Gamma$.
We denote by 
$$i(\Fol,c)$$
the measure of a path $c\subset \Sigma$
with respect to $\Fol$.
The intersection of $\Fol$ with $\gamma\in\Gamma$ is defined as 
\[i_\Fol(\gamma):=\inf_c i(\Fol,c)\]
where the infimum is taken over all closed loops $c$ transverse to
$\Lam$ and freely homotopic to $\gamma$.
We denote by
\[I_\Fol:=i(\Fol,\cdot)\colon \Gamma\to\R_+\]
the corresponding intersection function.

This may be reformulated in terms of pseudometrics.
We denote by $d_{\Fol}$ the pseudometric
on the universal cover $\Sigmat$ of $\Sigma$
associated with the lift $\Folt$ of $\Fol$:
it is defined by the formula
\[d_\Fol(x,y) :=\inf_c i(\Folt, c)\]
where the infimum is taken on  paths $c$
joining $x$ to $y$ in $\Sigmat$ and transverse to $\Folt$.
The length function $L_d\in (\R_+)^\Gamma$
of a pseudometric $d$ on a $\Gamma$-space $X$ is defined as
\[L_d(\gamma):=\inf_{x\in X} d(x,\gamma x)
  \; .\]
By definition, the length function of $d_\Fol$ on $\Sigmat$ is  the
intersection function of $\Fol$:
\[L_{d_\Fol}=I_\Fol.\]
%
The dual tree of a measured foliation $\Fol$ is defined as the quotient metric space 
$T(\Fol):=\Sigmat/d_\Fol$ of $\Sigmat$ by the pseudometric $d_\Fol$.
The tree $T(\Fol)$ inherits an action of $\Gamma$ by isometries,
with associated length function
$$L_{T(Fol)}=I_\Fol \;.$$

\subsection{Measured geodesic laminations}\label{ss-ml}
Let $\mLam=(\Lam,\nu)$ be a measured lamination
on $\Sigma$.
As for measured foliations, we denote by $i(\mLam, c)$
the $\mLam$-measure of a path $c\colon [0,1] \to \Sigma$
transverse  to $\Lam$. 
%
The intersection of $\mLam$ with $\gamma\in\Gamma$ is defined as 
 \[i(\mLam, \gamma):=\inf_c i(\mLam,c)\] 
where the infimum is taken over all closed loops $c$ transverse to
$\Lam$ and freely homotopic to $\gamma$,
and we denote by
\[I_\mLam:=i(\mLam,\cdot)\colon \Gamma\to\R_+\]
the corresponding intersection function.

We denote by $d_\mLam$ the  pseudometric
on $\Sigmat$  associated with $\mLam$:
this is the pseudometric defined on $\Sigmat-\Lamot$,  where
$\Lamot$ is the set of atomic\footnote{the leafs with
	positive weight} leafs
of the lift $\mLamt=(\Lamt,\nu)$ of $\mLam$,
and
by the formula
\[d_\mLam(x,y) :=\inf_c i(\mLamt, c)\]
where the infimum is taken on  paths $c$
joining $x$ to $y$ in $\Sigmat$ and transverse to $\Lamt$.
By definition, the length function of $d_\mLam$ on $\Sigmat-\Lamot$ is  the
intersection function of $\mLam$:
\[L_{d_\mLam}=I_\mLam.\]

We  refer to \cite{Kap09} for the classical correspondence between measured foliations and measured geodesic laminations on hyperbolic surfaces.
\subsection{Geodesic currents}\label{ss-gc}
We refer to
\cite{Bonahon88} or  \cite[\S 8]{Martelli}
for the background material in the case of closed geodesic surfaces,
%
%
and for
instance to \cite{BIPP19} for the generalisation to finite type surfaces.

A \emph{geodesic current} on a finite type hyperbolic surface $\Sigma=\mathcal H^2/\Gamma$
is a flip-invariant $\Gamma$-invariant positive Radon measure on the
space
of unoriented, uparametrized geodesics of $\mathcal H^2$,
that may be identified with the space 
\begin{equation*}
\calG(\mathcal H^2):=\{(x,y)\in(\partial\mathcal H^2)^2:\,x\neq y\}
\end{equation*}
of distinct pairs of points in the boundary at infinity $\partial\mathcal H^2$
of $\mathcal H^2$.
%
A basic example is the current $\delta_c$
associated to a closed geodesic $c$ in $\Sigma$, which is defined as
the sum of the Dirac masses on the  lifts of $c$ to $\mathcal H^2$.
%
Recall that the Bonahon intersection $i(\mu,\nu)$  of two geodesic currents
$\mu,\nu$ is defined as  the $(\mu\times\nu)$-measure
of any Borel fundamental domain for the $\Gamma$-action on the space
$\calDG(\mathcal H^2)\subset \calG(\mathcal H^2)\times\calG(\mathcal H^2)$ of pairs  of transverse
geodesics $(g,h)$.

We refer for instance to \cite[\S 8.3.4]{Martelli}
for the bijective correspondence
$\mLam \mapsto \mu_ \mLam$
between measured geodesic laminations $\mLam=(\Lam,\nu)$
of full support  $\Lam$ and  geodesics currents $\mu$ on
$\Sigma$ with $i(\mu,\mu)=0$,  equivalently currents $\mu$ such that no two geodesics in
the support of $\mu$ intersect transversally.
The geodesic current $\mu_\mLam$ has support $\Lam$, and,
for each geodesic arc $c$ transverse to $\Lam$,
the restriction of $\mu_\mLam$ to the set of geodesics $g$ of $\Lam$
intersecting $c$ is the pullback of the $\mLam$-mesure on $c$ by the
map $g \mapsto g\cap c$.
%
The notions of intersection then coincide as
\begin{equation*}
i(\mu_\mLam,\delta_c)=i(\mLam,c)\,,
\end{equation*}
for all closed geodesic $c$ in $\Sigma$.
Note that the union of two measured geodesic laminations with
disjoint support correspond to the sum of the associated currents.
From now on we will freely identify $\mLam$ with $\mu_\mLam$ whenever convenient.

Given a geodesic current $\mu$ on $\Sigma$, and a geodesic subsurface $\subS$ of $\Sigma$, we denote by 
$\mu|_{{\subS}}$ the restriction of $\mu$ to the subsurface $\subS$, namely the geodesic current 
$$\mu_{{\subS}}:=\chi_{\calG(\subS)}\mu,$$
where $\chi_{\calG(\subS)}$ is the characteristic function of the set $\calG(\subS)$ of geodesics whose projection lies in $\subS$. In general a geodesic current of full support might restrict to the zero current on a proper subsurface (this is the case when $\mu$ is the Liouville current of a hyperbolic structure), but if, for every boundary component $c$ of $\subS$, $i(\mu,\delta_c)=0$, then for every $\gamma\in\pi_1(\subS)$, $i(\mu,\delta_\gamma)=i(\mu|_{\subS},\delta_\gamma)$ \cite[Proposition 4.13]{BIPP21}.  

\subsection{Dual tree of a measured geodesic lamination}\label{ss-dt}
\label{s-DualTreeOfMLam}

  Let $\mLam=(\Lam,\nu)$ be a measured geodesic lamination on
  an hyperbolic surface $\Sigma$.
  We now recall the construction of the associated dual
$\R$-tree.
We follow the construction of \cite[\S 11.12]{Kap09}.


We first get rid of the atoms blowing-up along atomic leafs:
for each isolated leaf $c$ of $\Lam$,
cut $\Sigma$ along $c$ and insert an
annulus $B(c)=c\times [0,1]$, foliated by the parallel
circles $c\times {t}$. 
We endow $\Sigmab$ with the locally CAT(0)  metric $\metb$ on
equal to the original metric $\met$ of $\Sigma$
outside the annuli $B(c)$ and to the flat metric on
$B(c)$.
This gives a locally CAT(0) surface $(\Sigmab,\metb)$
homeomorphic to $\Sigma$,
with a geodesic  lamination $\Lamb$.
%
We call $(\Sigmab,\Lamb)$ the {\em blow-up}  of  $(\Sigma,\Lam)$.
%
The {\em blow-up}  of the transverse measure $\nu$ on $\Lam$ is
the non-atomic transverse measure $\nub$ on $\Lamb$,
obtained from $\nu$ by giving to each foliated annulus $B(c)$ the transverse
measure $\nu(c)dt$ on $[0,1]$.
%
%
%
The pseudometric $d_{\mLamb}$ associated to the measured lamination $\mLamb:=(\Lamb,\nub)$
is then a continuous and  everywhere-defined path
pseudometric on $\widetilde{\Sigmab}$.
%
%
 %
 The {\em $\R$-tree dual  to $\mLam$} is defined as the quotient
metric space $T(\mLam):=\Sigmabt/d_\mLamb$, whose metric will be
denoted by $d_\mLam$.

It is easy but crucial to see that
the geodesics of $T(\mLam)$ are the projections of the $\metb$-geodesics of
$\Sigmabt$.
%
In the following lemma, and in the rest of the article,
when we write \emph{geodesic} we mean 
a path $t\mapsto c(t)$ that is
additive for the distance $d$, namely such that
$d(c(t_1),c(t_3))=d(c(t_1),c(t_2))+d(c(t_2),c(t_3))$ 
for all $t_1\leq t_2\leq t_3$. 
Observe that $c$ may then be constant on some subset of its domain of definition, and needs not be parametrized at constant speed.

\begin{Proposition}
  \label{prop- proj in dual tree preserves geods}
Let  $c\colon I\to \Sigmabt$ be a $\metb$-geodesic. 
Then $c$ is a minimizing curve for  $d_{\mLamb}$, and
the projection of  $c$ on the tree $T(\mLam)$
is a geodesic.
\end{Proposition}

\begin{proof}
Let $x,y$ be points of $\Sigmabt$ and  $c\colon [0,1]\to\Sigmabt$ be the constant
speed geodesic segment from
$x$ to $y$.
As the leafs of $\Lamtb$ are geodesics,
and CAT(0) spaces are uniquely geodesic,
the path $c$ cannot cross twice the
same leaf of $\Lamtb$,
hence $i(\mLamb,c)$ is minimal and $d_{\mLamb}(x,y)=i(\mLamb,c)$.
Denote by $p\colon \Sigmabt\to T(\mLam)$ the canonical projection.
The  projection
$\ov{c}=p\circ c$ of $c$ in the dual tree $T(\mLam)$
does not backtrack, hence is a geodesic.
\end{proof}

The tree $T(\mLam)$ inherits an action of $\Gamma$ by isometries,
with associated length function
$$L_{T(\mLam)}=I_\mLam \;.$$
\compl{Comes from the fact that
$\ell_{T(\mLam)}(\gamma)=\ell_{(\Sigmabt,d_\mLamb)}(\gamma)=i(\mLamb,\gamma)=i(\mLam,\gamma)$.}

Note that when $\Sigma$ is closed the tree $T(\mLam)$
is then \emph{minimal} for the action of $\Gamma$,
namely there is no invariant proper subtree.
\compl{Beware not true for our lamination pieces !}%
%
Indeed  an invariant proper subtree  will lift
as a proper closed invariant convex subset in $\Sigmabt$,
and taking the closure in the CAT(0) compactification  we will then obtain a proper closed
invariant subset of the boundary at infinity $\partial_\infty\Sigmabt \simeq \partial\mathcal H^2$,
\compl{NB: uses that the cone from a point to $\partial_\infty\Sigmabt$ is
  $\Sigmabt$, not true when there is boundary}
which is impossible.

Note that the tree $T(\mLam)$ is essentially determined by
its length function ; more generally
we will use the length rigidity of actions on minimal trees with
length functions in $\ML$:
It is easily seen that the the action of $\Gamma$ on $T(\mLam)$ is
\emph{irreducible}, namely it has no global fixed point in $T\cup
\partial_\infty T$,
where
$\partial_\infty T$ denotes the boundary at infinity of $T$. 
This depens in fact only on the length function: the action is reducible if and only if its length function is of the form $\gamma \mapsto
|h(\gamma)|$, where $h\colon \Gamma\to \R$ is a homomorphism (see
\cite[Corollary 2.3]{CuMo87}).
In particular all minimal $\Gamma$-trees with length function
in $\ML$ are irreducible.
It follows then from the length rigidity for minimal irreducible
$\Gamma$-trees \cite[Theorem 3.2]{CuMo87}
that if $T$ and $T'$ are any two minimal $\Gamma$-trees
with the same length function in $\ML$,
then there is a unique  equivariant isometry $T\to T'$.

\subsection{The  $\R^n$-tree dual to a $n$-measured geodesic lamination.}\label{ss-Rntree}
\label{s-DualnTree}

Let $\n\in\N$.

\begin{Definition}
  A {\em $\R^n$-tree} is a $\R^n$-metric
  space  $(X, \vd=(d_i)_{i=1,\ldots,n})$ which is a $\R$-tree for the
  associated $\ell^1$-metric $d^1(x,y):=\sum_i d_i(x,y)$.
\end{Definition}
Observe that, while we only assume that  $d_i$ are pseudodistances, we require that their sum $d^1$ is a distance, namely it separates points. 
\begin{Definition}
A {\em $\n$-measured lamination} $\vmLam=(\Lam,\vnu)$ on
a surface
$\Sigma$ is   a lamination $\Lam$ on  $\Sigma$,
endowed with a $\n$-tuple
$\vnu=(\nu_1,\ldots, \nu_\n)$ of transverse measures of \emph{full joint support},
namely  $\Lam$ is the union of the supports $\Lam_i$ of
the measures $\nu_i$ for $i=1,\ldots,\n$.

It can equivalently be seen as a $\n$-tuple 
of {\em parallel} measured
laminations $(\mLam_1,\ldots, \mLam_\n)$,
that is, such that any two
leafs  of $\mLam_i$ and  $\mLam_j$ are either disjoint
or equal.
\end{Definition}

When $\Sigma$ is endowed with a locally CAT(0) metric,
A $\n$-measured geodesic lamination $\vmLam=(\Lam,\vnu)$
is called {\em geodesic} when $\Lam$ is geodesic.

When $\Sigma$ is a closed hyperbolic surface,
seeing measured geodesic laminations as
geodesic currents,
a $n$-tuple $(\mLam_1,\ldots, \mLam_\n)$
of measured geodesic laminations
is parallel if and only if
$$i(\mLam_i, \mLam_j)=0 \text{ for all } i,j$$
or equivalently that 
$$\mLam:=\sum_{i=1}^\n\mLam_i\text{ is a measured lamination.}$$ 

Let $\vmLam=(\Lam,\vnu)$ be a $\n$-measured lamination on
a closed hyperbolic surface $\Sigma$.
We now construct the associated dual $\R^\n$-tree,
adapting the construction of Section~\ref{s-DualTreeOfMLam}.

We take the $\R$-tree $T(\mLam):=\Sigmabt/d_\mLamb$ dual  to the measured
lamination $\mLam=(\Lam,\nu)$ where $\nu=\sum_i\nu_i$.

Recall from Section~\ref{s-DualTreeOfMLam}
that  $(\Sigmab,\Lamb)$ is the blow-up of  $(\Sigma,\Lam)$,
which is a locally CAT(0) surface $(\Sigmab,\metb)$
homeomorphic to $\Sigma$,
with a geodesic  lamination $\Lamb$,
and that $\mLamb=(\Lamb,\nub)$ where $\nub$ is the non-atomic transverse measure on $\Lamb$
obtained by blowing up $\nu$,
and that $d_{\mLamb}$ the associated
continuous path pseudometric on $\Sigmabt$.

Taking the blow-ups $\nub_i$ of the transverse measure $\nu_i$, we
obtained a $n$-tuple $\vnub:=(\nub_i)_i$ of non-atomic transverse measures on
$\mLamb$ (not necessarily of full support),
namely a $n$-measured lamination $\vmLamb=(\Lamb,\vnub)$ on
$\Sigmab$, which we will call the {\em blow-up}  of $\vmLam=(\Lam,\vnu)$.
%

Each of the measured laminations $\mLamb_i=(\Lamb,\nub_i)$ then
induces a continuous everywhere-defined
path pseudometric $d_{\mLamb_i}$ on $\widetilde{\Sigmab}$.

\begin{Definition} The {\em $\R^n$-tree $T(\vmLam)$ dual  to $\vmLam$} is defined as
the $\R$-tree $T(\mLam):=\Sigmab/d_\mLamb$ dual  to the measured
lamination $\mLam$,
endowed with the $\R^n$-pseudometric given by the $n$-tuple
\[\vd_\mLam:=(d_{\mLam_i})_{i=1,\dots,n}\]
of quotient pseudometrics $d_{\mLam_i}$ induced by $d_{\mLamb_i}$.
\end{Definition}
 It inherits an action of $\Gamma$ preserving $\vd_\mLam$, that
 is preserving each pseudometric $d_{\mLam_i}$.

 \begin{Proposition}\label{prop-isometry}
   $(T(\mLam),\vd_\mLam)$ is an $\R^n$-tree with associated
   $\ell^1$-metric $\lVert\vd_\mLam\rVert_1=d_\mLam$. 
 \end{Proposition}

 \begin{proof}
   The crucial fact is that, by Proposition~\ref{prop- proj in dual
     tree preserves geods},
   the infima involved
   in the definition of the pseudometrics $d_{\mLam_i}$
   are in fact all realized simultaneously for a same path $c$ in
   $\Sigmabt$, the $\metb$-geodesic.
In particular $c$ is a minimizing curve  for  $d_{\mLamb}$ and its on $T(\mLam)$
is a minimizing path for $\vd_\mLam$.
This  ensures that $d_{\mLam}=\sum_i d_{\mLam_i}$ on $T(\mLam)$.
%





The separation of $\vd_\mLam$ on $T(\mLam)$ follows:
if $\vd_\mLam(x,y)=0$ then 
 $d_{\mLam}(x,y)=\sum_i d_{\mLam_i}(x,y)=0$, hence $x=y$.
\end{proof}

\begin{Remark}
Proposition~\ref{prop-isometry} is special for the $\ell^1$-metric, and doesn't work for other distances, not even locally. For example, consider a minimal lamination $\lambda$ that supports two mutually singular measured laminations. Then the distance function $\lVert\vd_\mLam\rVert_2$ is not geodesic, not even locally. 
\end{Remark}

\section{The boundary of $\CharMax(\Gamma,\PSL(2,\R)^n)$ is $\PP((\ML)^n)$}
\label{s-Boundary}

In this section we prove
Theorem~\ref{thmintro-IdBoundary}, that identifies $\bordC{\Teich^n}$ with  $ \PP((\ML)^n)$, and deduce some consequences for other length spectra compactifications. 

\subsection{Proof of Theorem~\ref{thmintro-IdBoundary}}
Theorem~\ref{thmintro-IdBoundary} is classical when $n=1$, and follows from the work of Thurston.
Thurston furthermore proved \cite[Expose 8]{FLP} that $\ML$ is
homeomorphic to $\R^{6g-6}$,
 that it's projectivization $\PP\ML$ is homeomorphic to $\mathbb S^{6g-7}$, and that the resulting compactification is homeomorphic to a closed ball \cite[8.3.13]{Martelli}. 
 Using the latter fact we can prove the following lemma, which ensures that
 given $L$ in $\ML$, one  can choose a sequence in $\Teich$ 
 converging to $[L]$ with a fixed scale sequence. It will be crucial il the proof of Theorem~\ref{thmintro-IdBoundary}.
 
 \begin{Lemma}
 	\label{lem-FixedScaleSequence}
 	Let $L\in \ML$.
 	Let $(\lambda_k)$ be any increasing diverging sequence of positive real numbers.
 	Then there exists
 	a sequence of maximal representations $\rho_{k}\colon \Ga \to \PSL(2,\R)$
 	such that $\frac{1}{\lambda_k} L_{\rho_{k}}$
 	converges to $L$ in $(\R_+)^\Gamma$.
 \end{Lemma}
 
 \begin{proof}
 	Given a finite generating set $S\subset \Gamma$ and an isometric action $\rho$ of
 	$\Gamma$ on a metric space $X$, 
 	the {\em minimal displacement of $\rho$ with
 		respect to the generating set $S$} is defined by:
 	$$\lambda_S(\rho):=\inf_{x\in X}\sqrt{\sum_s d(x,\rho(s)x)^2}\;.$$
 	This defines a non zero proper and continuous
 	function
 	$\lambda_S \colon  \Teich \to \R_+$ 
 	(this is for example proven - in a much more general context -
 	in \cite[Prop. 25]{APrepcr}).
 	
 	We may suppose that $L$ is non-zero (otherwise we can take a constant
 	sequence $\rho_k$).
 	Let $T$ be a  $\R$-tree with minimal $\Gamma$-action, with length function $L$.
 	Let $D$ 
 	be the minimal displacement of $\Gamma$  on
 	$T$ with respect to the generating set $S$.
 	
 	Fix a point $[\rho_0]$ in the Teichm\"uller space $\Teich$.
 	As the compactification $\ov{\Teich}$ 
 	of $\Teich$ is a closed ball, 
 	there exists a
 	path $r(t)$, $t\in[0,1]$ from $[\rho_0]$ to $[L]$ in $\ov{\Teich}$
 	such that $r(t)$ belongs to $\Teich$ for $t\in[0,1[$.
 	Let $\rho_t$ be a representation  with $[\rho_t]=r(t)$.
 	Then, as the map 
 	$$\begin{array}{ccc}
 		[0,1[&\to&\R_+\\
 		t&\mapsto &\lambda_S(r(t))
 	\end{array}$$
 	is continuous and diverges as $t$ goes to $1$, there exists an increasing sequence $(t_k)_{k\geq K}$ in $[0,1[$ with limit $1$ such that $\lambda_S(r(t_k))=D\lambda_k$.
 	Since $[\rho_k]=r(t_k)$ converges to $[L]$ in $\ov{\Teich}$,
 	we have that $\frac{1}{\lambda_S(\rho_k)} L_{\rho_k}$
 	converges to $sL$ for some $s\in\R_+$.
 	Taking the asymptotic cone of this sequence (see for example \cite{APcomp}), we also have that $sL$ is the length function
 	of an action of $\Gamma$ on a real tree $T_\omega$ with minimal
 	displacement $1$ with respect to $S$.\footnote{see proof of Thm 5.6 in \cite{APcomp}} 
 	Let $T'\subset T_\omega$ be the minimal invariant
 	subtree. 
 	As $T'$ is a convex subset of $T_\omega$, the minimal
 	displacement of $\Gamma$ in $T'$ is the same as in $T_\omega$.
 	By length rigidity of actions on minimal
 	irreducible trees 
 	the trees $T$ and $\frac{1}{s}T'$ are equivariantly isometric,
 	hence have same minimal displacement  $s=\frac{1}{D}$. 
 	So we have that $\frac{1}{\lambda_k}L_{\rho_k} \to L$ as wanted.
 \end{proof}


We now have the ingredients needed to prove Theorem~\ref{thmintro-IdBoundary}, which we recall for the reader's convenience:
\begin{Proposition}
	\label{prop-IdBoundaryII}
	The  boundary of $\CharMax(\Ga,\PSL(2,\R)^n)=\Teich^n$ in 
	the Weyl chamber length compactification is
	$\PP((\ML)^n)$.
\end{Proposition}

\begin{proof}
It is easy to see that $\bordC(\Teich^n)\subset \PP((\ML)^n)$.
Let indeed $\rho_k=(\rho_{k,i})_{i=1,\ldots,n}$, $k\in\bN$, be a sequence 
in 
$\Hom(\Gamma,\PSL(2,\R)^n)$, which we identify with
$\Hom(\Gamma,\PSL(2,\R))^n$.
We write a nonzero function $\vL\colon \Gamma\to \R^n$ as
$\vL=(L_i)_{i=1,\ldots,n}$ with $L_i\colon \Gamma \to \R$.
The sequence of conjugacy classes $[\rho_k]$ converges to the
homothety class $[\vL]$ in 
the Weyl chamber length compactification if and only if there exists a 
sequence of positive real numbers $\lambda_k\to\infty$ ({\em scale
  sequence}) such that
the renormalized Weyl chamber length function 
$\frac{1}{\lambda_k}\vL_{\rho_k}$ 
converges to $\vL$, that is if 
$\frac{1}{\lambda_k}L_{\rho_{k,i}}$ 
converges to $L_i$ for all $i=1,\ldots,n$.
Then either $L_i=0$ or  $[\rho_{k,i}]$ converges to $[L_i]$ in
the length compactification $\ov{\Teich}$ of $\Teich$.
As $\bord\Teich=\PP(\ML)$, 
we have that each $L_i$ belongs to $\ML$, hence $[\vL]$ belongs to $\PP((\ML)^n)$.


The converse implication is a consequence of Lemma~\ref{lem-FixedScaleSequence}.
Let $\vL=(L_i)_{i=1,\ldots,n}\in(\ML)^n$.
Using Lemma~\ref{lem-FixedScaleSequence} with scale sequence
$\lambda_k=k$,  we can construct for each $i=1,\ldots,n$
a sequence of maximal representations 
$\rho_{k,i}:\Ga \to \PSL(2,\R)$, $k\in\N$, 
whose renormalized length function $\frac{1}{k}L_{\rho_{k,i}}$ 
converges to the length function $L_i$ in $(\R_+)^\Gamma$ as $k\to
\infty$.
\compl{\Anne I reverted to converges to $L_i$ in $(\R_+)^\Gamma$ as
  ``converges to $x$ in $X$'' is more precise for me than
  ``converges to $x\in X$'' (means convergence in  $X$ as a
  topological space)}
We now consider for $k\in\N$ the product representation 
$\rho_k\colon =(\rho_{k,i})_{i=1,\ldots,n}\colon \Gamma \to \PSL(2,\R)^n$, 
which, being a product of maximal representations, is a maximal
representation.
Then
$\frac{1}{k}\vL_{\rho_k}=(\frac{1}{k}L_{\rho_{k,i}})_{i=1,\ldots,n}$
converges to $(L_i)_{i=1,\ldots,n}=\vL$ as $k\to \infty$. 
Hence $[\vL]$ belongs to $\bordC(\Teich^n)$.
\end{proof}

%
As a result we deduce that $\PP((\ML)^n)\subset
\PP((\R_+^\Gamma)^n)$ is a sphere of dimension $n(6g-6)-1$, being
the topological join of  $n$ spheres:
\begin{Corollary}
		The  boundary of $\CharMax(\Ga,\PSL(2,\R)^n)=\Teich^n$ in 
	the Weyl chamber length compactification is
	homeomorphic to $\mathbb S^{n(6g-6)-1}$.
\end{Corollary}

\subsection{Applications to other length compactifications}

As an application of Proposition~\ref{prop-IdBoundaryII} we can also understand various other compactifications of $\CharMax(\Ga,\PSL(2,\R)^n)$:

\begin{Corollary}\label{c-Wlength}
	For any $W$-invariant norm $N\colon \R^n\to \R_+^n$, the  boundary of
	$\CharMax(\Ga,\PSL(2,\R)^n)$ in 
	the  $N$-length compactification is
	the projectivization $\PP(N((\ML)^n))$ of the image $N((\ML)^n)$ of
	$(\ML)^n$ by the map sending $\vL\colon \Gamma\to \R^n$ to
	$N\circ \vL\colon \Gamma\to \R$.
\end{Corollary}
	In particular, the  boundary of $\Teich^n$ in 
	the $\ell^1$-length compactification is
	the projectivization $\PP(\sum_{i=1}^n\ML)$ of the space of geodesic currents
	that can be decomposed as the sum of $n$ measured laminations.

\begin{Remark}
While the map $(\ML)^n\to (\sum_{i=1}^n\ML)$ has fibers of cardinality 2 on the set of geodesic currents with positive systole, the fiber over minimal measured laminations (and more generally over geodesic currents that admit a  Bonahon-orthogonal decomposition $\mu=\mu_1+\mu_2$ with $\mu_1$ a minimal measured lamination) is higher dimensional. As a result it is, in general, not clear how to determine the topology of $\PP(\sum_{i=1}^n\ML)$. Indeed it follows from the arguments in \cite[Theorem 2]{DLR} that no finite set of simple closed curves is sufficient to separate points in $\PP(\sum_{i=1}^n\ML)$. This is in strong contrast to $\PP((\ML)^n)$: there is a collection of $9g-9$ simple closed curves on $\Sigma$ whose $\R^n$-length function already separate points in $\PP((\ML)^n)$. 
\end{Remark}

\section{$\R^2$-mixed structures}\label{sec:mix}
The purpose of this section is to prove Theorem~\ref{thmIntro1} from the introduction, which interprets points in $\ML\times\ML$ as length functions of mixed structures. For this we introduce the notion of flat structures in \S\ref{s-Flat},  mixed structures in \S\ref{s-Mix}, and their associated length functions in \S\ref{s-LengthofMix} and prove the result in \S\ref{s-MixAndML2}.
\subsection{Flat structures}
\label{s-Flat}
We will consider flat structures on a punctured finite type surfaces, namely the complement of a finite set of marked points considered as punctures in a compact topological surfaces. When dealing with mixed structures, the finite type surfaces will typically be obtained from geodesic subsurfaces of the original surface $\Sigma$ by collapsing each boundary components to a cusp point.
\begin{Definition}
A \emph{ half-translation structure} $\K$ on a finite type surface $\Sigma$ 
is 
a $(\R^2\rtimes\bZ/2\bZ,\R^2)$-structure   
on $\Sigma$, with conical singularities of angle $k\pi$, $k\geq 2$,
extendible at punctures, with possibly angle $\pi$ singularities
at the punctures.
\end{Definition}
With a slight abuse of notation we denote by $\Flat(\Sigma)$ the moduli space of half-translation structures on $\Sigma$, where two such structures are identified if they are isotopic.
Note that the flat structures we consider here  are \emph{directed} (that is, with a preferred
vertical direction).
\begin{Remark}\label{DLR-Flat}
Duchin-Leiniger-Rafi \cite{DLR}, as well as Ouyang-Tamburelli \cite{OT1} consider, instead, the space of flat structures on $\Sigma$, where they identify isometric marked structures. As a result we have fibrations
$$\Flat(\Sigma)\to\FlatDLR(\Sigma)\to\Teich(\Sigma).$$
The fiber of $\Flat(\Sigma)\to\FlatDLR(\Sigma)$ is the circle $\mathbb S^1$, which acts on a half-translation surface by rotation. Since our structures are marked, and any half-translation structure induces a conformal structure on the surface, the space $\Flat(\Sigma)$ fibers over the Teichm\"uller space.  We will never need this fact, but it is well known that the fiber over $X$ in this fibration identifies with the space of holomorphic quadratic differentials over the Riemann surface $X$.
\end{Remark}

Let $\K$ be a flat structure on $\Sigma$, and let
 $\Fol_1$, $\Fol_2$
be the vertical and horizontal  measured foliations of $\K$.
This is a pair of transverse measured foliations.
For $i=1,2$, 
we denote by 
$$\ell_{\K,i}(c)=i(\Fol_i,c)=\int_c\abs{dx_i}$$
and  by $d_{\K,i}:=d_{\Fol_i}$ the associated pseudometric on $\Sigmat$
This defines the natural $\R^2$-metric
\[\vd_\K=(d_{\K,i})_i\]
on $\Sigmat$.
%
We denote $L_{\K,i}:=I_{\Fol_i}\colon \Gamma\to\R$ the corresponding length
function. 
%
Let $T(\Fol_i):=\Sigmat/d_{\K,i}$ be the dual tree 
of $\Fol_i$, and $p_i\colon \Sigmat\to T(\Fol_i)$ the corresponding projection.

The universal cover $\Sigmat$ of $\Sigma$ is a CAT(0) metric space for the 
flat metric  associated to $\K$, which we will denote by
$d_{\CAT(0)}$.
Note that $d_{\CAT(0)}$ is not in general equal to  the
metric $\|\vd_\K\|_2$ induced from the $\R^2$-metric $\vd$ by taking
the $\ell^2$-norm, nevertheless it is the associated  length metric.
We denote $\Sigmatc$ the completion of $\Sigmat$ with respect to
$d_{\CAT(0)}$,
which consists in adding one fixed
point $x_{\ct}$ for each parabolic subgroup 
$\Gamma_{\ct}$ of $\Gamma=\pi_1(\Sigma)$
(corresponding to lifts $\ct$ of punctures $c$ of $\Sigma$),
see for example \cite[Lemma 7.2]{Mor}. 
The flat structure extends on $\Sigmatc$.
It is easy to see that $\Fol_i$, $d_{\K,i}$  and $p_i$ extend naturally 
to the completion $\Sigmatc$ of $\Sigmat$.
%
%

We define the {\em $\ell^1$-length metric}
  on  $\Sigmat$, 
as the Finsler metric $d_\K$ induced by the
$\ell^1$-norm $\norm{x}_1=\abs{x_1}+\abs{x_2}$  on $\R^2$.
This metric is clearly equivalent
to the CAT(0) metric $d_{CAT(0)}$, in particular
extends to $\Sigmatc$.
We denote by
\[\L_\K(\gamma):=L_{d_\K}(\gamma)\]
the corresponding length of $\gamma\in\Gamma$.
%
%
%
We now establish some basic properties of the $\ell^1$-metric
that we will need.
%
Recall that in the following lemma, a geodesic is an additive  
 path $t\mapsto c(t)$ for the distance $d$.

\begin{Lemma}
\label{lem-L=L1+L2Flat}
Let $\K$ be a flat structure on a finite type surface $\Sigma$. 
The following properties hold:
\begin{enumerate}
\item 
\label{it:ProjGeodsFlat}
For any CAT(0) geodesic $c\colon I\to \Sigmatc$, the projections
$c_i=p_i\circ c$ of $c$ in the dual trees $T(\Fol_i)$ are geodesics; 

\item
\label{it:d^1=d_1+d_2Flat}
We have $d_\K=d_{\K,1}+d_{\K,2}=\|\vd_\K\|_1$ on $\Sigmatc$.
In particular, the  CAT(0) geodesics are geodesics for $d_\K$;

\item 
\label{it:L=L_1+L_2Flat}
We have $L_\K(\gamma)=L_{\K,1}(\gamma)+L_{\K,2}(\gamma)$ for all $\gamma\in\Gamma$.
\end{enumerate}
\end{Lemma}

\begin{proof}
We prove (\ref{it:ProjGeodsFlat}). Let $i\in\{1,2\}$.
As $T(\Fol_i)$ is a tree, we only need to prove that the path $c_i$ does not
backtrack.
Let $t_1\leq t_2\leq t_3$ be real parameters in $I$ 
such that $c_i(t_1)=c_i(t_3)$. Then the points
$c(t_1)$ and $c(t_3)$ are in a common leaf of the foliation $\Fol_i$.  
As the leafs
of the foliation are geodesic for the CAT(0) metric $d$ on $\Sigmat$, we
can deduce, by uniqueness of geodesics in CAT(0) metric spaces, 
that the point $c(t_2)$ is
on the same leaf, hence that $c_i(t_1)=c_i(t_2)=c_i(t_3)$. This proves
that $c_i$ is  a geodesic in the tree $T(\Fol_i)$.
We now prove (\ref{it:d^1=d_1+d_2Flat}).
The $\ell^1$-length of a path $c$ is, by definition,
$\ell^1(c)=\int_c\abs{dx_1}+\int_c\abs{dx_2}
=\ell_1(c)+\ell_2(c)$. 
We then clearly have that
$d(x,y)\geq d_1(x,y)+d_2(x,y)$.
On the other hand, if $c$ is the CAT(0) geodesic in $\Sigmatc$ from $x$ to $y$, we know by (\ref{it:d^1=d_1+d_2Flat})
that both projections $c_i=p_i\circ c$ of $c$ are
geodesic. So for $i=1,2$ we have 
$d_i(x,y)=d(p_i(x), p_i(y))=l(c_i)=\ell_i(c)$ and hence
 $\ell(c) = d_1(x,y)+d_2(x,y)$.

We finally prove (\ref{it:L=L_1+L_2Flat}).
As $d=d_1+d_2$ we clearly have 
$L_\K(\gamma)\geq L_{\K,1}(\gamma)+L_{\K,2}(\gamma)$.
If $L_\K(\gamma)=0$ then we are done.
Otherwise $\gamma$ has no fixed point in $\Sigmatc$ and, 
since $\Sigmatc/\Gamma$ is compact,
there is a CAT(0) geodesic $c$ in $\Sigmatc$ translated by $\gamma$.
By (\ref{it:d^1=d_1+d_2Flat}) both projections $c_i=p_i\circ c$ of $c$ are
geodesics translated by $\Gamma$. 
So, for any $x$ on $c$, and $i=1,2$
$$d_i(x,\gamma x)=d_i(p_i(x), \gamma p_i(x))=L_{ T(\Fol_i)}(\gamma)=L_{\K,i}(\gamma).$$
Then $d(x,\gamma x)=L_{\K,1}(\gamma)+L_{\K,2}(\gamma)$.
\end{proof}

\subsection{$\R^2$-mixed structures on a surface}
\label{s-Mix}
In this section we introduce a natural notion of $\R^2$-mixed structure on a
surface. This generalizes flat structures, and refines the notion of
mixed structure introduced by Duchin-Leininger-Rafi and Morzadec 
\cite{DLR, Mor}. The definition follows the point of view of
\cite{DLR}, see Section~\ref{s-DualTreeGradedSpace} for the
metric viewpoint analoguous to \cite{Mor}. 
%

\begin{Definition}
A  {\em $\R^2$-mixed structure} on a compact hyperbolic surface
$\Sigma$ with boundary
is a triple $\M=(\subS,\K,\vmLam)$
where 
\begin{itemize}
\item $\subS$ is  a (possibly disconnected, possibly empty) 
  open geodesic subsurface of $\Sigma$, namely
a union of complementary components of a collection of disjoint simple
closed geodesics.
%
\item $\K$ is a flat structure  on $\subS$ extendible at punctures (when
$\subS$ is compactified as a punctured surface),
\item 
$\vmLam=(\Lam, \vnu)$ is a {\em $2$-measured geodesic lamination} on
$\Sigma-\subS$, that is 
$\Lam$ is a geodesic lamination on  $\Sigma$, included in $\Sigma-\subS$, and 
$\vnu=(\nu_1,\nu_2)$ is a  pair of transverse measures on $\Lam$
of full
support, that is  $\Lam=\supp(\nu_1)\cup\supp(\nu_2)$.\footnote{We allow the supports of $\nu_i$ to overlap, but don't require that one of the $\nu_i$ is fully supported}
\end{itemize}
\end{Definition}
This imposes topological restrictions on the subsurface $\subS$:
no connected component of $\subS$ can be a pair of pants, since a pair of pants doesn't support any non-trivial flat structure.

\begin{Example}\label{ex:mix}
We illustrate an example of an $\R^2$-mixed structure on a surface of
genus $3$ in Figure~\ref{fig:mix}. In this case the support of the
lamination $\Lam$ consists of the three colored curves $c_1$, $c_2$,
$c_3$,
each with the given pair of weights $(x_i,y_i)\in\R_+^2-\{(0,0)\}$.
Observe that in this example $\supp(\nu_1)=\{c_1,c_2\}$ and $\supp(\nu_2)=\{c_2,c_3\}$. 
The subsurface $\subS$ is the disjoint union of a thrice punctured torus $\subS_1$, bounded by the curve $c_0$ and the curve $c_1$, and a twice punctured torus $\subS_2$, endowed with the flat structures $\K_1, \K_2$ illustrated above in the picture. In the flat pictures the parallel sides are identified, and the colored points correspond to punctures (corresponding to the curves, in $\Sigma$, of the same color). In particular  the two black punctures $p_0,p_0'$ in $\K_1$  come from the same curve. Observe that, while $c_0$ is a boundary component of $\subS$, it doesn't belong to the lamination $\Lam$, since the curve is not contained in the support of neither transverse measure.
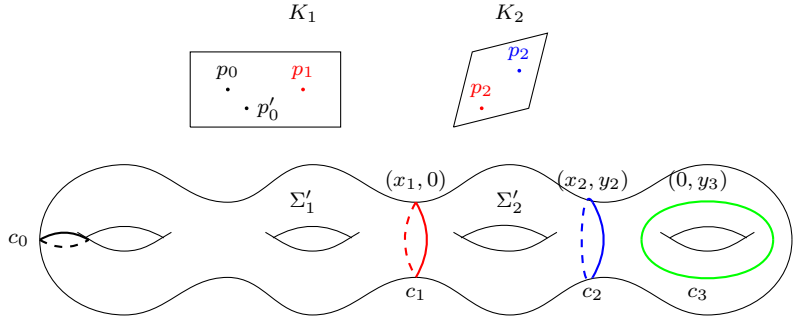
\begin{figure}[h]
\begin{tikzpicture}[scale=.5]
		\draw(-5,0) [out=90, in=180] to (-2.75,2)[out=0, in=180] to (0,1) to (2.25,2)[out=0, in=180] to (5,1) to (7.5,2) to (10,1) to (12.75, 2) to [out=0, in=90](15,0);
		\draw (-5,-0)  [out=-90, in=180]to (-2.75,-2) [out=0, in=180]to (0,-1) to (2.25,-2) to (5,-1) to (7.5,-2) to (10,-1) to (12.75,-2) to [out=0, in=-90](15,0);
			\draw (-4,0.2) [out=-30, in=180]to (-2.75,-.3)[out=0, in=210] to (-1.5,0.2);
		\draw (-3.8,0.05) [out=30, in=150] to (-1.7,0.05);
		
		\draw (1,0.2) [out=-30, in=180]to (2.25,-.3)[out=0, in=210] to (3.5,0.2);
		\draw (1.2,0.05) [out=30, in=150] to (3.3,0.05);
		
		\draw (6.,0.2) [out=-30, in=180]to (7.5,-.3)[out=0, in=210] to (8.75,0.2);
		\draw (6.2,0.05) [out=30, in=150]to  (8.5,0.05);
		
		\draw (11.5,0.2) [out=-30, in=180]to (12.75,-.3)[out=0, in=210] to (14,0.2);
		\draw (11.75,0.05) [out=30, in=150] to (13.8,0.05);
		
		
		\draw (-5,0) [thick, black, out=30,in=150]to (-3.7,0);
        \draw (-5,0) [thick,black, dashed, out=-30,in=-150]to (-3.7,0);
        \node at (-5,0) [left] {\tiny $c_0$};
		
		\draw (5,1) [thick, red, out=-60,in=60]to (5,-1);
		\draw (5,1) [thick,red, dashed, out=-120,in=120]to (5,-1);
		\node at (5,1) [above ] {\tiny $(x_1,0)$};
		\node at (5,-1) [below] {\tiny $c_1$};
				
		\draw (9.7,1) [ blue,thick, out=-60, in=60] to (9.7,-1);
		\draw (9.7,1) [ blue,thick,dashed, out=120, in=-120] to (9.7,-1);
		\node at (9.7,1) [above] {\tiny $(x_2,y_2)$};
		\node at (9.7,-1) [below] {\tiny $c_2$};
		
		\draw(11,0) [ green,thick, out=90, in=90] to (14.5,0);
    	\draw(11,0) [ green,thick, out=-90, in=-90] to (14.5,0);
	    \node at (12.5,1) [above ] {\tiny $(0,y_3)$};
	    \node at (12.5,-1) [below] {\tiny $c_3$};
	\node at (2,1)  {\tiny $\subS_1$};

	\node at (7.5,1)  {\tiny $\subS_2$};
	
	\node at (2,5.5) [above] {\tiny$\K_1$};
	\draw (-1,3) to (3,3) to (3,5) to (-1,5) to (-1,3);
	\filldraw [red] (2,4) circle [ radius=1pt]   node [above] {\tiny $p_1$}; 
	\filldraw [black] (0,4) circle [ radius=1pt]   node [above] {\tiny $p_0$}; 
	\filldraw [black] (0.5,3.5) circle [ radius=1pt]   node [right] {\tiny $p_0'$};

		\node at (7.5,5.5) [above] {\tiny$\K_2$};
	\draw (6,3) to (8,3.5) to (8.5,5.5) to (6.5,5) to (6,3);
\filldraw [blue] (7.75,4.5) circle [ radius=1pt]   node [above] {\tiny $p_2$}; 
\filldraw [red] (6.75,3.5) circle [ radius=1pt]   node [above] {\tiny $p_2$}; 




		
		
		
\end{tikzpicture}
\caption{An $\R^2$-mixed structure on a surface of genus 4.}\label{fig:mix}
\end{figure}

\end{Example}
We denote by $\Mix(\Sigma)$ the moduli space of $\R^2$-mixed structures, where two $\R^2$-mixed structures are identified if the subsurfaces and the laminations agree and the flat structures are equivalent, namely isotopic to each other.

\begin{Remark}\label{r-DLRmix}
Following up on Remark~\ref{DLR-Flat}, observe 	that in the $\R^2$-mixed structures we consider here, the flat part is always directed. Forgetting the vertical direction we obtain a map $\Mix(\Sigma)\to \MixDLR(\Sigma)$, where $\MixDLR(\Sigma)$ denotes the mixed structures studied in \cite{DLR}. In general the map $\Mix(\Sigma)\to \MixDLR(\Sigma)$ has fibers $(\mathbb S^1)^k$ where $k$ denotes the number of connected components of $\Sigma'$. In \cite{DLR} the set $\MixDLR(\Sigma)$ is identified with the corresponding set of geodesic currents, a perspective that is generalized in \cite{OT0,OT1}, where the flat metrics associated to cubic (resp. quartic) differentials are considered.
\end{Remark}	
\subsection{The $\R^2$-length function of a mixed structure}
\label{s-LengthofMix}
Let $\M=(\subS,\K,\vmLam)$ be a $\R^2$-mixed structure 
on $\Sigma$. 
%
Denote by $(\Fol_1,\Fol_2)$ the vertical and horizontal measured
foliations on $\subS$ associated with the flat structure $\K$. 
%
\label{s-PairofmLamofMix}
Let $\mLam_{\K,i}$ be the measured geodesic 
lamination on $\subS$ corresponding to $\Fol_i$, namely, with the same
intersection function.
It can be seen as a measured geodesic lamination in $\Sigma$.
Let
$\mLam_{\M,i}$ be the measured geodesic lamination on $\Sigma$ obtained by taking the
union of $\mLam_{\K,i}$  %
and of the measured geodesic lamination $\mLam_i=(\Lam_i, \nu_i)$ on
$\Sigma-\subS$, where $\Lam_i:=\supp(\nu_i)$. 
Regarding measured geodesic laminations as geodesic
currents on $\Sigma$, we have
$$\mLam_{\M,i}:=\mLam_{\K,i}+\mLam_i$$

The {\em $\R^2$-length} of $\gamma\in\Gamma$ with respect to the mixed structure
$\M$ is then defined as the pair:
$$\vL_\M(\gamma):=(i(\mLam_{\M,i},\gamma))_{i=1,2}.$$
We denote $L_{\M,i}(\gamma):=i(\mLam_{\M,i},\gamma)$ the factors.
%
Note that if $\subS=\Sigma$ then the length $L_{\M,i}$ agrees with $L_{\K,i}$.

This gives a map  $$\vL\colon \Mix(\Sigma)\to \ML\times \ML.$$
We will show in the next subsection that this is indeed a bijection.

\subsection{Interpretation of $\ML\times\ML$ as $\Mix(\Sigma)$}
\label{s-MixAndML2}
The goal of the section is to provide an identification of $\ML\times\ML$, the boundary $\dCharMaxWL(\Ga,\PSL(2,\R)^n)$ for the case $n=2$, with $\Mix(\Sigma)$ thus obtaining a geometric interpretation of the former space. The main ingredient for this is the following application of  the decomposition result of \cite{BIPP19} to a sum of
 measured geodesic laminations (seen as geodesic currents), a result that works for general $n$:
 \begin{Proposition}\label{p.dec}
	Let $\vL=(L_i)_{i=1,\ldots,n}\in(\ML)^n$, and denote by $\mu_i$ 
	the measured lamination with intersection function $L_i$. Then
	there exists a canonical open geodesic subsurface $\subS$, such
	that
        each $\mu_i$ decomposes as a disjoint union $\mu_i=\mu^i_{\subS}\cup\mu^i_\Lambda$
        of measured laminations,
        and  
	\begin{enumerate}
        \item $\mu^i_{\subS}$ is supported in $\subS$, $\mu^i_\Lambda$ in its complement;
        \item  $(\mu^i_{\subS})_i$ have positive joint systole 
      	\item $(\mu^i_\Lambda)_i$ are parallel, namely the union of their support is a geodesic lamination $\Lambda$ (contained in $\Sigma\setminus\subS$).
        \end{enumerate}
\end{Proposition}

  \begin{proof}
  Let $\mu_i=(\Lam_i,\nu_i)$ be the measured lamination  
  on $\Sigma$ corresponding to $L_i$, that is such that
  $i(\mu_i,\cdot)=L_i$
  on $\Gamma$.
 
  As in \cite{BIPP19} we consider the collection $\Ee=\Ee_\mu$ of \emph{closed $\mu$-short solitary geodesics}. This is the collection, canonically associated to $\mu$, of  simple closed geodesics in $\Sigma$ that have $0$-intersection with $\mu$, and that do not intersect any geodesic that doesn't intersect the support of $\mu$. 
  Applying the decomposition theorem  \cite[Corollary 1.9]{BIPP19}
  to the geodesic current 
   $\mu=\sum_{i=1}^n\mu_i$,
  we get that the surface decomposes along 
  the collection $\Ee=\Ee_\mu$ 
  in a finite number of open connected subsurfaces with geodesic boundary
  $\Sigma-\Ee=\bigcup_{v\in\Vv} \Sigma_v$,
  and the current $\mu$ decomposes as a sum
  $$\mu=\sum_{v\in\Vv}\mu_v 
  + \sum_{c\in\Ee} t_c \delta_c$$
  where
  $\mu_v=\mu|_{\Sigma_v}$ is the  restriction of $\mu$ to $\Sigma_v$ (recall \S\ref{ss-gc}).
  
  %
  %
  Furthermore, for every $v\in \Ee$ 
  for which $\mu_v\neq 0$ precisely one of the following holds:
  \begin{enumerate}
  \item either $\Syst_{\Sigma_v}(\mu_v)>0$,
  \item or $\mu_v$ is a measured lamination.
  \end{enumerate}
  where 
   $$\Syst_{\Sigma_v}(\mu_v):=\inf \{i(\mu_v,c) |\,c\subset\mathring\Sigma_v \text{ closed geodesic }\}\;.$$
 
  Let $i\in\{1,\ldots,n\}$. 
  We first see that 
  no closed geodesic $c\in \Ee$  intersects transversally the
  support of $\mu_i$: indeed, since $i(\mu,c)=i(\mu_1,c)+i(\mu_2,c)=0$, we have $i(\mu_i,c)=0$.
  So the measured geodesic lamination $\mu_i$ decomposes as the union of
  measured geodesic laminations $\mu_{v,i}$ 
  included in the open subsurface $\Sigma_v$, and possibly
  closed leafs $c\in\Ee$ with transverse measure 
  $t_{c,i} \in \R_+$.
  That is, seeing all those measured geodesic laminations as geodesic currents on $\Sigma$:
  $$\mu_i=\sum_{v\in\Vv}\mu_{v,i} 
  + \sum_{c\in\Ee} t_{c,i} \delta_c$$
 
  By uniqueness of the decomposition of $\mu$ along $\Ee$, 
  for each $v\in\Vv$ we have  
  $$\mu_v=\sum_i \mu_{v,i}$$
  and for each $c\in\Ee$
  $$t_c=\sum_i t_{c,i}$$
 
  Let $v\in\Vv$.
  If $\mu_v$ is a measured lamination, then the measured laminations
  $\mu_{v,i}$, $i=1,\ldots,n$
  are parallel: indeed, as $i(\mu_v,\mu_v)=\sum_{1\leq i,j\leq
    n}i(\mu_{v,i},\mu_{v,j})=0$,
  and $i(\mu_{v,i},\mu_{v,j})\geq 0$,
  we get that $i(\mu_{v,i},\mu_{v,j})=0$ for all $i,j\in \{1,\ldots,n\}$.
  \end{proof}
 
We can now prove the main result of the section.
\begin{Proposition} 
\label{prop-MLxMLtoMix}
Any $\vL=(L_1,L_2)$ in $\ML\times\ML$  is the $\R^2$-length function 
of a  unique 
$\R^2$-mixed structure $(\subS,\K,\vmLam)$ on $\Sigma$.
%
\end{Proposition}

\begin{proof}

As above we denote by $\mu_i$ the geodesic current corresponding to the lamination $(\Lambda_i,\nu_i)$ with length function $L_i$. We use Proposition~\ref{p.dec} to decompose $\mu_i:=\mu^i_{\subS}+\mu^i_\Lambda$. 

For any connected component $\Sigma_v$ of the positive systole subsurface $\subS$, we 
denote by
 $\Gamma_v=\pi_1(\Sigma_v)$ and 
$L_{v,i}\colon \Gamma_v\to \R_+$ the intersection function of 
$\mu_{v,i}$ on
$\Sigma_v$. Then
$L_{v,i}$ is the intersection function $I_{\Fol_i}$ 
of a measured foliation $\Fol_i$ on the surface $\Sigma_v$. 
%
Denote $C=\Syst_{\Sigma_v}(\mu_v)$.
We have that, for all non parabolic $\gamma$ in $\Gamma_v$,
$$\max (i(\Fol_1,\gamma), i(\Fol_2,\gamma))\geq \frac{C}{2} >0.$$
It is known
(see for example \cite[Theorem 7]{GaWa12})
%
%
that the two measured foliations $\Fol_1, \Fol_2$ are transversely realisable, 
that is, up to
replacing $\Fol_i$ by an equivalent measured foliation 
(an operation that doesn't change the length
function), 
they arise as the vertical and horizontal measured
foliation
of  a flat structure $\K_v$ on the surface $\Sigma_v$.
In particular the measured geodesic laminations on $\Sigma_v$ associated with $\K_v$ are
$$\mLam_{\K_v,i}=\mu_{v,i}$$
for $i=1,2$. 
We denote by $\K$ the flat structure on $\subS$ equal to $\K_v$ on each $\Sigma_v$. 

Let $\mLam_i$  be the measured geodesic lamination 
obtained by taking the
union of $\mu_{v,i}$, for the $v$ such that $\mu_v$ is a lamination, 
and of the closed geodesics  $c\in\Ee$ with weight $t_{c,i}$.
 Then 
$\vmLam=(\mLam_1,\mLam_2)$ is a pair of parallel measured geodesic laminations on
$\Sigma$ with $\lambda_\mu=\lambda_1+\lambda_2$, and
 $\M=(\subS,\K,\vmLam)$ is a mixed structure on $\Sigma$, 
with associated pair of laminations
$$\mLam_{\M,i}=\mLam_{\K,i}+\mLam_i=\mu_i.$$
In particular taking intersection functions on $\Gamma$
we have $L_i=L_{\M,i}$ for each $i$.

The uniqueness of $\M$ (up to isotopy of the flat part) is given by
the uniqueness of the decomposition in Corollary 1.9 of \cite{BIPP19} and by
the injectivity of the natural map from quadratic differentials to
pairs of equivalence classes of measured foliations
\cite[Theorem 3.1]{GaMa91}.
\end{proof}


\section{The  $\R^2$-tree-graded space dual to a mixed structure}
\label{s-DualTreeGradedSpace}

In this section, we construct
the {\em $\R^2$-tree-graded space $X_\M$ dual}  to a $\R^2$-mixed structure  
$\M=(\subS,\K,\vmLam)$ on a closed hyperbolic surface $\Sigma$. In \S\ref{ss-treegrad} we recall the definition of tree-graded space, and discuss a general construction to glue pseudometrics which we use in \S\ref{ss-constXM} to construct the tree graded space $X_M$ dual to a mixed structure $\M$. In \S\ref{ss-proofXM} we prove that that space is indeed tree graded, ad discuss its relevant geometric properties. 

\subsection{Generalities on $\R^n$-tree-graded spaces.}\label{ss-treegrad}

Recall from \cite{DrSa05} the notion of tree-graded metric space%
\footnote{Observe, however, that in \cite{DrSa05} the tree-graded spaces are additionally assumed to be complete. We relax this assumption because the tree-graded space associated to a mixed structure will, in general, not complete: it is well known that the $\R$-tree dual to a measured lamination $\lambda$ is only complete when $\lambda$ has no minimal component, and these $\R$-trees are examples of tree graded spaces associated to mixed structures.}:

\begin{Definition}
	A geodesic metric space $(X,d)$ is {\em tree-graded}
	with respect to a collection $\Pieces$ of geodesic subsets called {\em pieces}
	if
	\begin{description}
		\item[(TG1)] Any two different pieces have at most one common point.
		\item[(TG2)] Any simple geodesic triangle in $X$ is contained in one piece.
	\end{description}
	Here a geodesic triangle is {\em simple} if its sides meet only in the respective endpoints.
\end{Definition}

We now adapt this to define an {\em tree-graded $\R^n$-space}, recall from Definition~\ref{d.Rnmetric} the notion of an $\R^n$-metric space:
%
\begin{Definition}
  A $\R^n$-metric space $(X, \vd=(d_i)_{i=1}^n)$
  is {\em tree-graded} with respect to a collection $\Pieces$ of subsets
  if $X$ is tree-graded (with respect to $\Pieces$) 
  for the associated $\ell^1$-metric $d:=\sum_{i=1}^n d_i$.
\end{Definition}

In the next subsection, we will construct the tree-graded space associated with a mixed
structure on a surface $\Sigma$ as a quotient of a blowup $\Sigmabt$
of $\Sigma$ by pseudometrics defined by gluing.
%
We now recall the  general construction of  a global pseudometric 
on a CAT(0) surface $\Sigmat$ obtained by gluing pseudometrics defined on
geodesic pieces, the initial ingredient for the construction in the next section.

Let $\Sigma$ be a locally CAT(0) surface, and
$\Ee$ be  a set of disjoint simple closed geodesics on $\Sigma$.
We denote by $\Eet$ the set of their lifts to the universal cover $\Sigmat$. 
A {\em piece}  of $\Sigmat$ is defined as the closure $\Pt$
of a complementary component $\Wt$ of $\Eet$ in $\Sigmat$.
Two different pieces are {\em adjacent} if they have non-empty intersction
(which is then a geodesic in $\Eet$ bounding each of the pieces).
We denote by $\calP(\Sigmat)$ the set of pieces of $\Sigmat$.

\begin{Definition}
A {\em chain} between two points $x,y$ in $\Sigmat$
is a sequence $C=(x_0=x,x_1,\ldots ,x_{k+1}=y)$ in $\Sigmat$ such that any
two consecutive points $x_j$ and $x_{j+1}$ are in a same piece
$\Pt_j$, with $\Pt_j\neq\Pt_{j+1}$.
\end{Definition}

 For $j=1,\ldots,k$ the pieces $\Pt_{j-1}$ and $\Pt_j$ are
then adjacent and  $x_j$ is on their common boundary geodesic $\ct_j$. 
 Such a chain defines a path in the simplicial tree dual to $\Eet$.
 We call the chain {\em straight} if this path is geodesic,
 that is if and only
 {if
 $\ct_{j-1}\neq \ct_{j}$ for $j=1,\ldots,k$}. Then
 $\ct_1,\ldots \ct_k$ is the ordered sequence of geodesics in $\Eet$
 separating $x$ and $y$ (going from $x$ to $y$).

Given a pseudometric $d^\Pt$  on each piece $\Pt$ of $\Sigmat$,
we define the $d$-length of a chain $C$ as 
$$\ell_d(C)=\sum_{j=0}^{k+1} d^{\Pt_j}(x_j,x_{j+1})$$ 
The induced pseudometric on $\Sigmat$ is then defined  by 
$$d(x,y)=\inf_C \ell_d(C)$$
where the infimum is taken over all chains $C$ joining $x$ to $y$.
It is easy to see that we may restrict to straight chains $C$.
If the restriction of $d^{\Pt}$ to the geodesics $\ct$ of $\Eet$ is $0$ for each piece $\Pt$, then the pseudo-distance $d(x_j,x_{j+1})$ does not  depend on the choice of $x_j$ on $\ct_j$, and thus all straigth chains have the same length.
In this case we may alternatively define $d(x,y)$ as the $d$-length of any
straight chain, thus in particular $d$  restricts to the original
pseudometric $d^\Pt$ on each piece $\Pt$.

\compl{Note further that any $\met$-geodesic
of $\Sigmat$ crosses the collection of
decomposing geodesics $\Eet$ in a straight chain.
In particular if the pseudometrics $d^\Pt$ are straight then $d$ is
also straight, namely $\met$-geodesics of $\Sigmabt$ are geodesics
for $d$.}

\subsection{Construction of the $\R^2$-space $X_\M$}\label{ss-constXM}
Let $\M=(\subS,\K,\vmLam)$  be a $\R^2$-mixed structure
on a closed hyperbolic surface $\Sigma$.
Here, as always,  $\vmLam=(\Lam,\nu_1,\nu_2)$. 


We first
 resolve the atoms in the lamination part, 
by taking the blow-up $(\Sigmab,\vmLamb)$ of $(\Sigma,\vmLam)$
as in Section~\ref{s-DualnTree}.
Recall from  Sections~\ref{s-DualTreeOfMLam} and
\ref{s-DualnTree}
that $\Sigmab$ denotes the CAT(0) surface obtained inserting  in $\Sigma$ a flat foliated  annulus $B(c)=c\times [0,1]$ at each isolated leaf $c$ of $\Lambda$, and $\Lamb$ is the associated lamination whose non-atomic transverse measure $\nub_i$ is 
obtained by extending $\nu_i$ with the transverse
measure $\nu_i(c)dt$ on $[0,1]$ for each foliated annulus $B(c)$.


Our next goal is to define a $\R^2$-pseudometric $\vd_\M$ on
$\Sigmabt$ associated with $\M$,
that is a pair  of pseudometrics $(d_{\M,i})_i$  on $\Sigmabt$.
We first define the pieces which we will glue as outlined in \S~\ref{ss-treegrad}.
%
We denote by $\subSb$ the open subsurface of
$\Sigmab$ corresponding to $\subS$,
and denote by $\Eeb$ its set of boundary geodesics in $\Sigmab$.
We denote by $\Eebt$ the set of their lifts to $\Sigmabt$. 
\begin{Definition}
A {\em piece}  of $\Sigmabt$ is the closure $\Pt$
of a complementary component $\Wt$ of $\Eebt$ in $\Sigmabt$.
We will call such a piece 
\begin{itemize}
	\item a {\em flat piece} if $\Wt$ projects in $\subSb$, and
	\item a {\em  lamination piece} otherwise\footnote{beware that the metric
  $\metb$ is hyperbolic on flat pieces.}.
\end{itemize}
\end{Definition}

\begin{Example}
	If $\M$ is the mixed structure described in Example~ \ref{ex:mix}, the locally CAT(0) surface $(\Sigmab,\metb)$ is obtained by endowing
        $\Sigma$ with an hyperbolic metric and then gluing
        three flat cylinders $C_i=[0,1]\times c_i$ to the hyperbolic surface
        $\Sigma\setminus\{c_1,c_2,c_3\}$.
        The pieces in its  universal cover $\Sigmabt$ have four
        types:
        flat pieces
        isometric to the universal covers of the completion of $\subS_1$ and  $\subS_2$ (up to now these pieces are endowed with a complete hyperbolic metric with geodesic boundary), lamination
        pieces of anular type, isometric to the universal cover of
        $C_1$, namely to the Euclidean strip $[0,1]\times\R\subset\R^2$, with measured
        geodesic lamination equal to the vertical foliation with 
        transverse measures $\nu_1=x_1dx$, $\nu_2=y_1dx$,
        and lamination pieces isometric the universal cover
        of $(\Sigma \setminus (\ov \subS \cup c_3))\cup C_2\cup C_3$.
       \end{Example}

We now define a new pseudometric on each piece. Let $i\in\{1,2\}$.
%
%
On each  lamination piece $\Pt$ of $\Sigmabt$,
we define $d^\Pt_{\M,i}$ as the
restriction to $\Pt$ of the pseudometric $d_{\mLamb_i}$ associated to
the non-atomic  mesured lamination $\mLamb_i:=(\Lamb,\nub_i)$.
%
On each flat piece $\Pt$ of $\Sigmabt$ we, instead, define the pseudometric $d^\Pt_{\M,i}$ as  intersection with the horizontal (resp. vertical) measured foliation on $K$. More precisely we consider the canonical projection
$\Pt=\ov{\Wt}\to \Wtc$ to the completion $\Wtc$ of $\Wt$ with respect to the CAT(0) metric
given by the flat structure $\K$, and we define the pseudometric $d^\Pt_{\M,i}$ as the pullback 
of the pseudo-distance $d_{\K,i}$ introduced in \S~\ref{s-Flat} through this projection. 
%
%
Note that  the restriction of $d^\Pt_{\M,i}$ to any boundary
geodesic of a piece $\Pt$ is always $0$.

We  define the pseudometric  $d_{\M,i}$ on $\Sigmabt$
as the gluing of the pseudometrics $d^\Pt_{\M,i}$ on the pieces
$\Pt\in\calP(\Sigmabt)$ as in Section ~\ref{ss-treegrad}.
We denote by
\[\vd_M=(d_{\M,i})_i\]
the corresponding $\R^2$-pseudometric on $\Sigmabt$, and by
%
$$d_\M:=d_{\M,1}+d_{\M,2}$$
the associated  $\ell^1$-pseudometric on $\Sigmabt$.
It follows from the construction that $d_\M$ is the gluing of the $\ell^1$-pseudometrics
$d^\Pt_{\M}=\sum_i d^\Pt_{\M,i}$
on the pieces $\Pt\in\calP(\Sigmabt)$.


\begin{Definition}
The {\em $\R^2$-metric space $X_\M$ associated with $\M$} is 
the quotient
$$X_\M:=\Sigmabt/d_\M$$
of $\Sigmabt$ by the $\ell^1$-pseudometric $d_\M$,
endowed with the $\R^2$-metric induced by $\vd_\M$,
that is the pair of  pseudometrics
induced by $d_{\M,i}$, $i=1,2$.
\end{Definition}
%
The action of $\Gamma$ on
$\Sigmabt$ induces an action of $\Gamma$ on $X_\M$ preserving the $\R^2$-metric $\vd_\M$.

\begin{Example} 
  The pieces in the  tree-graded $\R^2$-space $X_\M$ associated to the mixed structure $\M$ described in Example ~\ref{ex:mix} have 4 isometry types:
	\begin{itemize}
\item flat pieces isometric
	to the completion of the universal cover of
	$\K_1\setminus\{p_1, p_0,p_0'\}$.
	\item flat pieces isometric
	to the completion of the universal cover of
 $\K_2\setminus\{p_1,p_2\}$, 
	\item closed segments of $\R^2$-length
	$(x_1,0)$,
	\item  $\R^2$-simplicial trees of infinite valence, obtain by attaching, to the dual tree  to the curve $c_3$ in the
	subsurface $\Sigma\setminus \subS$, a segment of $\R^2$-length $(x_2,y_2)$ to the fixed points of the elements of $\Gamma_\Sigma$ corresponding to the curve $c_2$. 
	\end{itemize}
Observe that, in particular, the tree associated to the  lamination pieces are, in this example, not minimal.
\end{Example}

\subsection{Basic properties of $X_\M$}\label{ss-proofXM}
In this section we prove that $X_\M$ is indeed tree-graded, that its induced  $\R^2$-length function corresponds to
the pair $(\mLam_{\M,1},\mLam_{\M,2})$ of measured geodesic
laminations associated with the mixed structure $\M$, and we discuss the isometry types of the pieces of $X_\M$.

Recall that we defined the laminations $\mLam_{\M,i}$
as the sum 
of  two measured geodesic laminations  on $\Sigma$
$$\mLam_{\M,i}:=\mLam_{\K,i}+\mLam_i$$
where 
$(\mLam_{\K,i})_{i=1,2}$ are the measured geodesic laminations
supported in $\subS$  induced
by the horizontal (resp. vertical) foliation of the flat structure $\K$  and $\mLam_i=(\Lam,\nu_i)$.

We denote by $L_{X_{\M,i}}$, $i=1,2$ the length function of $d_{\M,i}$ on
$X_\M$,
namely
\[L_{X_{\M,i}}(\gamma):=\inf_{x\in X_\M}d_{\M,i}(x,\gamma.x) \;.\]

We denote by $\piXM\colon \Sigmabt\to X_\M$ the canonical projection. 
The 
{\em pieces} of $X_\M$ are the images $X_\Pt:=\piXM(\Pt)$ in $X_\M$ of the pieces $\Pt$ of $\Sigmabt$.
We denote by $\Pieces(X_\M)$ the set of pieces of $X_\M$.
When a lamination piece $\Pt\subset \Sigmabt$ does not meet
the lamination $\Lamb$, then the
corresponding piece $X_\Pt$ is reduced to a point.
We will call the space $(X_\M,\vd_\M=(d_{\M,1},d_{\M,2}))$  the 
{\em tree-graded  $\R^2$-space dual to $\M$}.

\begin{Proposition}
  \begin{enumerate}
  \item
    \label{it-TreeGraded}
    The $\R^2$ metric space $(X_\M,\vd_\M)$ is  tree-graded  with
    respect to  $\Pieces(X_\M)$.

  \item
    \label{it-LamPieces}
    If $\Pt$ is a lamination piece, 
    $(X_\Pt, \vd_\M)$ 
    is isometric to the image of $\Pt$ in the  $\R^2$-tree 
    $T(\vmLam):=\Sigmab/d_\mLamb$ dual to the $2$-measured geodesic
    lamination $\vmLam$. In particular $X_\Pt$ is then a $\R^2$-tree.

  \item
    \label{it-FlatPieces}
    If $\Pt$ is a flat piece, then 
    $X_\Pt$ is isometric to the complete flat surface $\Wtc$, 
    the pseudometrics
    $d_{\M,i}=d_{\K,i}$   are the vertical/horizontal pseudometrics, 
    and $d_\M$ is  the $\ell^1$-metric $d_\K$
     (see \S~\ref{s-Flat}).

  \item \label{it-LengthsCorrespond}
    For all $\gamma\in\Gamma$,
    \begin{equation*}
      L_{X_{\M,i}}(\gamma)=i(\mLam_{\M,i},\gamma)=L_{\M,i}(\gamma)
    \end{equation*}
  \end{enumerate}
\end{Proposition}

\begin{proof}
  The projection $\piXM\colon \Sigmabt\to X_\M$
  sends
	each geodesic  $\ct\in\Eebt$ to a point $x_\ct$ of $X_\M$.

We first prove (\ref{it-LamPieces}) and (\ref{it-FlatPieces}). 
By construction, the piece $X_\Pt$ is the biggest Hausdorff quotient of
$\Pt$ with respect to the $\ell^1$-pseudometric $d_\M$, 
endowed with the quotient pseudometrics $d_{\M,i}$. As a result (\ref{it-FlatPieces}) holds.

To check (\ref{it-LamPieces}) observe that if $\Pt$ is a lamination piece, 
since $d_{\M,i}$ is the
restriction of $d_{\mLamb_i}$, then $d_\M$ is the restriction of the pseudometric
$d_\mLamb$ 
induced  by the measured lamination $\mLamb=(\Lamb,\nub_1+\nub_2)$ on $\Sigmab$, 
and  $(X_\Pt, d_{\M,1}, d_{\M,2})$ 
identifies with the image of $\Pt$ in the  $\R^2$-tree 
$T(\vmLam):=\Sigmab/d_\mLamb$ dual to the $2$-measured geodesic lamination $\vmLam$ 
(see Section~\ref{s-DualnTree}).
%
\compl{Note that the action of $\Gamma_\Pt$ on $X_\Pt$ is not necessarily minimal, due to
	possible segments fixed by boundary elements corresponding to
	atoms.}%

We now prove (\ref{it-TreeGraded}).
Since all pieces $X_\Pt$ are geodesic for the
$\ell^1$-metric, it follows easily that $(X_\M,d_\M)$ is geodesic
(a geodesic between $x$ and $y$ is obtained by concatenating geodesics in
the separating pieces).

	We first prove  (TG1), by showing  that
	if two pieces $X_\Pt, X_{\Pt'}$ meet,
	then $X_\Pt\cap X_{\Pt'}=x_\ct$ where $\ct$ is the  boundary component
	of $\Pt$ separating $\Pt$ from $\Pt'$:
	we denote by $\ov{x}$ the image in $X_\Pt$ of a point $x$ in $\Pt$.  Suppose that 
	$\ov{x}=\ov{y}$ with $y\in\Ppt$.  
	Then  
	\begin{equation}\label{e-d0}
		d_\M(x,y)=0=\sum_{j=0}^{k+1} d^{\Pt_j}_\M(x_j,x_{j+1})
		\end{equation} 
	where
	$(x_0,\ldots,x_{k+1})$ is a straight chain from $x$ to $y$. Recall that this means that
	$x_0=x$, $x_{k+1}=y$ and, for all $j=1,\ldots,k$, the point $x_j$ belongs to the geodesic $\ct_j$, where  
	$(\ct_j)_{j=1,\ldots, k}$ is the ordered sequence  of geodesics in $\Eebt$
	separating $x$ and $y$, and we denote by  $\Pt_j$ the piece containing $x_j$ and $x_{j+1}$.
	Then $\ct_1=\ct$ and Equation \eqref{e-d0} implies that
        $d^{\Pt_0}_\M(x,x_1)=0$ proving
	that $\ov{x}=x_\ct=\ov{y}$.

	We now verify (TG2).
	Consider a simple non-trivial geodesic triangle
        with vertices $\ov{x}, \ov{y},\ov{z}$ in $X_\M$, and lift its vertices to
	three points $x,y,z$ in $\Sigmabt$. If $x,y,z$ are not in a common
	piece of $\Sigmabt$, then there is a geodesic $\ct$ in $\Eebt$
	separating one of the three points, say $x$, from the others. Then the
	corresponding point $x_\ct$ in $X_\M$ lies on each geodesic from $\ov{x}$ to
	$\ov{y}$ and on each geodesic from $\ov{x}$ to
	$\ov{z}$ in $X_\M$. Hence either the geodesic triangle is not simple or $x_\ct=\ov x$. In the second case we can change representative, and obtain a geodesic triangle in $\Sigmabt$ with shorter lengths (for the CAT(0) distance). Since the elements in $\Eebt$ are uniformly separated, the process terminates and shows that we can find preimages in the same piece.
We conclude proving (\ref{it-LengthsCorrespond}).
  We have that
  $L_{d_{\M,i}}(\gamma)=\inf_{x\in \Sigmabt}d_{\M,i}(x,\gamma.x)$.
  Let $x\in\Sigmabt$.
  Let $c$ be  the geodesic segment from $x$ to $\gamma.x$
  in $\Sigmabt$ for the CAT(0) metric $\metb$.
  It crosses the decomposing geodesics $\Eebt$ in  a
  straight chain $(x_0=x,x_1,\ldots ,x_{k+1}=\gamma.x)$,
  hence
  $$d_{\M,i}(x,\gamma.x)=\sum_{j=0}^{k}
  d_{\M,i}^{\Pt_j}(x_j,x_{j+1})\; .$$
  (see \S~\ref{ss-treegrad}).
  Let $J_F$ be the set of $j\in J=\{0,\ldots,k\}$ such that $\Pt_j$ is a
  flat piece
  and $J_L$ be the set of $j\in J$ such that $\Pt_j$ is a
  lamination piece.
  Recall that  the measured geodesic
  lamination on $(\Sigmab,\metb)$  corresponding
  to $\mLam_i$ is  $\mLamb_i=(\Lamb,\nub_i)$.
  Denote by $\mLamb_{\K,i}$ the measured geodesic
  lamination on $(\Sigmab,\metb)$ induced by
  $\mLam_{\K,i}$ (which is included in the flat pieces).
  
  As the geodesic lamination $\mLamb_i$ is supported on lamination pieces,
  we have
  \begin{align*}
    \sum_{j\in J_L}d_{\M,i}^{\Pt_j}(x_j,x_{j+1})
    &= \sum_{j\in J_L}d_{\mLamb_i}(x_j,x_{j+1})\\
    &=\sum_{j\in J}d_{\mLamb_i}(x_j,x_{j+1})\\
    &=d_{\mLamb_i}(x,\gamma.x)\\
    &\geq i(\mLamb_i,\gamma) \;.
  \end{align*}
  Similarly, since for any $x,y$ in (possibly different components of) the {boundary} $\partial\Pt$ of a flat piece
  $\Pt$ it holds
  $d_{\M,i}^{\Pt}(x,y)=d_{\mLam_{\K,i}}(x,y)$, and  the geodesic lamination $\mLamb_{\K,i}$ is supported on flat
  pieces, 
  we have
  \begin{align*}
    \sum_{j\in J_F}d_{\M,i}^{\Pt_j}(x_j,x_{j+1})
      &=\sum_{j\in J}d_{\mLamb_{\K,i}}(x_j,x_{j+1})\\
      &=d_{\mLamb_{\K,i}}(x,\gamma.x)\\
      &\geq i(\mLamb_{\K,i},\gamma).
  \end{align*}
  Hence:
  \begin{align*}
    d_{\M,i}(x,\gamma.x)
    &\geq  i(\mLamb_{\K,i},\gamma)+i(\mLamb_i,\gamma)\\
    &=i(\mLam_{\K,i},\gamma)+i(\mLam_i,\gamma)\\
    &= i(\mLam_{\M,i},\gamma).
  \end{align*}
  If $x$ is on  an axis of $\gamma$ in the CAT(0)
  surface $(\Sigmabt,\metb)$, then there is equality.
\end{proof}

\section{Embeddings in products of trees}
\label{s-EmbeddingInProductOfTrees}

If $X$ is  a product $X=X_1\times \cdots X_\n$ of metric spaces
$(X_i,d_i)$, the {\em $i^{th}$-factor pseudometric} is the pseudometric
$d_i(x,y)=d_i(x_i,y_i)$
obtained by pulling back the metric on $X_i$ via the canonical
projection. The {\em $\ell^1$-metric} on $X$ is the metric given by
$d=\sum_i d_i$.

\begin{Proposition} 
  \label{prop-EmbeddingInProductOfTrees}
  Consider a $\R^2$-mixed structure $\M$ on $\Sigma$
  with $\R^2$-length function $\vL$, and
  a pair of isometric actions of $\Gamma$ on  
  $\R$-trees  $T_1,T_2$ with  length functions
  $(L_{T_1},L_{T_2})=\vL$.
  Let $(X_\M,(d_1,d_2))$ be the tree-graded $\R^2$-space associated with $\M$.
  There is 
  an equivariant embedding 
  $$f\colon X_\M \mapsto T_1\times T_2$$ 
  preserving each factor pseudometric $d_i$.
  In particular, $f$ is isometric for the $\ell^1$-metric $d=d_1+d_2$.
\end{Proposition}

\begin{Example}
	If $\M$ is a flat surface, equivalently	$\subS=\Sigma$, the associated tree-graded $\R^2$-space is the flat surface given by the universal cover $\wt\M$ of $\M$. The map $f$ gives an embedding of $\wt\M$ 
	in any product of trees with the correct length function. Such embedding is isometric for the
	$\ell^1$-metric, and bilipschitz for the CAT(0) metric
	(compare \cite[Exemple 4, \S 2.3]{Gui04}). Such embedding is never isometric for the CAT(0) metric.

	%
    If, instead, $\M$ is a $2$-lamination, namely in the cases in which 
	$\subS=\emptyset$, $(X_\M,d)$ is a tree. The map $f$ gives an equivariant embedding of this tree  in any product of trees that induce the correct length function. This map is  is isometric if the product of trees is endowed with the
	$\ell^1$-metric.  Observe, however, that there are laminations $\Lambda$
        that support mutually singular transverse measures. For these
        the image of the $\R^2$-tree  won't be a geodesic subset of $T_1\times T_2$ if $T_1\times T_2$ is endowed with the CAT(0) metric.
          In general, when $T_1\times T_2$ is endowed with the CAT(0) metric, it is possible to show that the embedding is isometric if and only if all the laminations are homothetic.
\end{Example}

\begin{proof}[Proof of Proposition~\ref{prop-EmbeddingInProductOfTrees}]
We first use rigidity of lengths in $\R$-trees to reduce to the case
where each
$T_i$ is the  tree $T(\mLam_{\M,i})$ dual to the the measured geodesic
lamination $\mLam_{\M,i}$
on $\Sigma$ associated with $\M$:
Let $T'_i$ denote the minimal subtree of $T_i$ invariant by
$\Gamma$. 
The actions of $\Gamma$ on the trees $T'_i$ and 
$T(\mLam_{\M,i})$
are minimal and have same length function, which belongs to $\ML$.
Hence there is an equivariant isometry 
$h_i\colon T(\mLam_{\M,i})\isomto T'_i\subset T_i$.
Then the diagonal map $h=(h_i)_i$ is an equivariant embedding 
$\prod_iT(\mLam_{\M,i}) \to \prod_i T_i$ preserving each factor pseudometric $d_i$.

We now construct the canonical map $p_i\colon X_\M\to  T(\mLam_{\M,i})$.
We denote $T_{\M,i}\colon X_\M/d_{\M,i}$ the biggest Hausdorff quotient of $X_\M$
with respect to the pseudometric $d_{\M,i}$ and by
$p_i\colon X_\M \to T_{\M,i}$ the corresponding projection.
The tree $T_{\M,i}$ can be identified with $T(\mLam_{\M,i})$:  
since $T_{\M,i}$ is tree-graded with $\R$-trees pieces
$p_i(X_\Pt)=\Pt/d_{\M,i}$, it is itself an $\R$-tree.
%
%
Since the action of $\Gamma$ on $\wt \Sigmab$ is
  minimal,
  \compl{In the sense there is no invariant closed convex subset}%
  the action on $T_{\M,i}$ is minimal as well.
The length function of this action is 
$$L_{X_\M,i}=L_{d_{\M,i}}=i(\mLam_{\M,i},\cdot)$$
In particular, 
by rigidity of length functions for actions on minimal
trees, there is an equivariant isometry
$T_{\M,i} \isomto T(\mLam_{\M,i})$.
Then the diagonal map $p=(p_i)_i$ from $X_\M$ to $\prod_i T(\mLam_{\M,i})$ 
sends each pseudometric $d_{\M,i}$ on the factor pseudometric $d_i$,
 hence  it takes 
the metric $d_\M=\sum_i d_{\M,i}$ to the $\ell^1$-metric $d=\sum_i d_i$ on $\prod_i T(\mLam_{\M,i})$.
The results follows 
taking the map $f=h\circ p$.
\end{proof}


\end{document}